\documentclass[11pt,reqno,UTF8,twoside]{amsart}
\usepackage{mathrsfs}
\usepackage{amsfonts,amssymb,amsmath,amsthm}
\usepackage{cite}
\usepackage[colorlinks=true,citecolor=red,linkcolor=blue]{hyperref}
\usepackage{titletoc}
\usepackage{geometry}
\usepackage{bm}
\usepackage{indentfirst}
\usepackage{graphicx}
\usepackage{stmaryrd}
\usepackage{tikz}
\usepackage{ulem}
\usepackage{color,soul}
\usepackage{setspace}
\usepackage[pagewise]{lineno}
\usepackage{float}

\usepackage{todonotes}

\def\T {{\mathbb T}}
\def\C {{\mathbb C}}

\def\R {\mathbb{R}}
\def\N {\mathbb{N}}
\def\Z {\mathbb{Z}}
\def\d{{\,\rm d}}
\numberwithin{equation}{section}
\newtheorem{theorem}{Theorem}[section]
\newtheorem{lemma}[theorem]{Lemma}

\newtheorem{proposition}[theorem]{Proposition}

\newtheorem{remark}[theorem]{Remark}
\theoremstyle{definition}

\newtheorem{definition}[theorem]{Definition}

\newcommand{\supp}{\operatorname{supp}}
\newcommand{\Avg}[2]{\langle #1\rangle_{#2}}
\def\poscals#1#2{\langle#1,#2\rangle}

\tikzstyle{idea} = [rectangle, rounded corners, minimum width=2cm, minimum height=1cm, text centered, draw=black, align=center]
\tikzstyle{process} = [rectangle, minimum width=3cm, minimum height=1cm, text centered, draw=black, align=center]
\tikzstyle{point} = [coordinate, on grid]
\tikzstyle{arrow} = [thick,->,>=stealth]
\tikzstyle{dasharrow} = [dashed,->,>=stealth]

\makeatletter
\@namedef{subjclassname@2020}{\textup{2020} Mathematics Subject Classification}
\makeatother

\title[Control of KdV on space-time measurable sets]{\Large  The periodic KdV with control on space-time measurable sets}

\author[Jingrui Niu]{Jingrui Niu}

\author[Ming Wang]{Ming Wang}

\author[Shengquan Xiang]{Shengquan Xiang}

\address[Jingrui Niu]{Sorbonne Université, CNRS, Université Paris Cité, Inria Team CAGE, Laboratoire Jacques-Louis Lions (LJLL), F-75005 Paris, France
}
\email{jingrui.niu@sorbonne-universite.fr}

\address[Ming Wang]{School of Mathematics and Statistics, HNP-LAMA, Central South University, Chang-sha, Hunan 410083, P.R. China}
 \email{m.wang@csu.edu.cn}

\address[Shengquan  Xiang]{School of Mathematical Sciences, Peking University, 100871, Beijing, China.}
\email{shengquan.xiang@math.pku.edu.cn}

\begin{document}

	

    \vspace{5mm}
\begin{abstract}
In this paper, we establish the local exact controllability of the KdV equation on torus around equilibrium states, where both the spatial control region and the temporal control region are sets of positive measure. The proof is based on a novel strategy for proving observability inequalities on space-time measurable sets. This approach is applicable to a broad class of dispersive equations on torus. 

\end{abstract}

    \maketitle

     \setcounter{tocdepth}{1}
 \tableofcontents

\section{Introduction}
In this article, we concentrate on the control problem of the Korteweg--De Vries (KdV) equation posed on the one-dimensional torus $\T=\R/2\pi\Z$, 
\begin{equation}\label{eq: CKdV-tori-intro}
\partial_tu(t,x)+\partial_x^3u(t,x)+u(t,x)\partial_xu(t,x)=f(t,x)\mathbf{1}_{E_T\times F}(t, x),\; (t,x)\in[0,T]\times\T,
\end{equation}
where $E_T\subset[0,T]$ and $F\subset\mathbb{T}$ are two measurable sets with positive measures. 

The KdV equation, since its derivation in 1895, has served as an important nonlinear dispersive model for waves on shallow water surfaces, and has been extensively studied in the literature from various perspectives. The well-posedness problem has been extensively investigated in the literature, notably in the works \cite{Bona-Smith,Bourgain-93-kdv,KPV96,CKSTT-03,KV-19}, among others.  

The results concerning the associated controllability properties are fruitful. We refer to the works \cite{Russell-Zhang-96, LRZ-2010, Rosier97,GG08,KX,Xiang-19} and the references therein, where the emphasize is placed on the corresponding observability inequality.    In addition, the KdV on the critical length intervals with nonlinear control has also attracted considerable attention. We refer the interested readers to \cite{CC04,Cerpa07,CC09,CKN,NX, nguyen2023}. We also refer to the surveys \cite{RZ09, Cerpa14} and the references therein.

\vspace{2mm}

The primary objective of this paper is twofold. Firstly, we aim to prove the exact controllability of the KdV equation \eqref{eq: CKdV-tori-intro} around equilibrium states with a control distributed on a space-time measurable set, instead of a usual control region $[0,T]\times\omega$, where $\omega\subset\T$ is open. Secondly, instead of applying the classic moment method, we establish the linear observability via a new approach inspired by the proof of Miheev's theorem from harmonic analysis, which is more suitable for the measurable settings. This approach is notably robust and applicable to establishing observability for a wide range of perturbed operators. 

\subsection{Mass-conserved controllability}
KdV equation governs the behavior of shallow water waves in a channel. It is known to possess an infinite set of conserved integral quantities, one of which is the \textit{total mass}:
\[
\int_{\T}u(t,x)\d x. 
\]
We study the KdV equation \eqref{eq: CKdV-tori-intro} with a distributed control input function $f(t,x)$ serving as a forcing source. From the historical origins of the KdV equation,  it is natural to keep the conservation law of the total mass.

For any function $\varphi\in L^1(\T)$, we define its average over the measurable set $A\subset\T$ via
\begin{equation}\label{eq: average-def}
\langle\varphi\rangle_{A}:=\left\{
\begin{array}{ll}
     \frac{1}{|A|}\int_{A}\varphi(x)\d x,&|A|>0,  \\
     0,&|A|=0. 
\end{array}
\right.
\end{equation}
To study the mass-conserved system, for any $M\in\R$, we introduce 
\begin{equation}\label{eq: defi-mass-L^2}
L^2_M(\T):=\{\varphi\in L^2(\T);\Avg{\varphi}{\T}=M\}.
\end{equation}
{\bf Definition} (Mass-conserved controllability) We say the KdV system \eqref{eq: CKdV-tori-intro} achieves the mass-conserved exact controllability if and only if for $T>0$ and $M\in\R$, there exists a control function $f$ such that for any $u_0,u_1\in L^2_M(\T)$, the KdV equation \eqref{eq: CKdV-tori-intro} admits a unique solution $u\in C([0,T];L^2(\T))$ satisfying that 
$u(0,x)=u_0(x) \mbox{ and }u(T,x)=u_1(x)$.

\vspace{2mm}
It is notable that in order to achieve the mass-conserved controllability, the control function $f$ must have a specific form. Given a measurable set $F\subset\T$ with  positive measure, let $g=\frac{1}{|F|}\mathbf{1}_{F}$. Then, for any function $h$, we consider the control function $f$ in forms of
\begin{equation}\label{eq: defi-f}
\mathcal{L}(h):=\frac{1}{|F|}\mathbf{1}_{F}(x)\left(h(t,x)-\frac{1}{|F|}\int_{\T}h(t,y)\mathbf{1}_F(y)\d y\right).
\end{equation}

The controllability problem asks for:
Given $T>0$, $M\in\R$ and two states $u_0,u_1$ in $L^2_M(\T)$, can one find a control input $f$ such that the equation \eqref{eq: CKdV-tori-intro} achieves the mass-conserved controllability?  

A first answer to this problem was provided in \cite{Russell-Zhang-96} by Russell and Zhang. They proved that when $E_T=[0,T]$ is the whole time interval, and $F=\omega$ is an open subset, one can find a control $f$ with $\mathrm{supp}f\subset [0,T]\times\omega$ to achieve the mass-conserved controllability locally. After that, in \cite{LRZ-2010}, Laurent, Rosier, and Zhang considered the stabilization problem related to \eqref{eq: CKdV-tori-intro} and obtained a global mass-conserved controllability when the control region is an open set. 

Recently, control problems involving measurable control regions have garnered significant attention. Extensive research has focused on parabolic models, including heat equations on bounded domains (see \cite{GWangJEMS14,BM23,HWWLog24,GLBMO}), heat equations with bounded potentials on $\R^n$ (see \cite{EgidiVes,WWZZ,WangZhang23}), and heat equations with potentials growing at infinity (see \cite{BKJaming21,Dickeuncertainty2023JFAA,Dickespectral2024,wang2024quantitative} and the references therein). In these works, both the time control region $E_T$ and the spatial control region $F$ can be   measurable sets.

In contrast, results for dispersive equations remain relatively scarce. To the best of our knowledge, existing works are limited to Schr\"odinger equations on $\T^2$ \cite{BurqZworski2019}, fractional Schr\"odinger equations on $\T$ \cite{alphonse2025}, Schr\"odinger equations on $\R$ \cite{SuSunYuanJFA25,HWW25} and on $\R^2$ \cite{Bal-Martin23}. Furthermore, in these cases, only the spatial control region $F$ is permitted to be a measurable set; the time region $E_T$ typically requires stronger assumptions, containing an interval.

This disparity naturally raises the question: For the KdV equation \eqref{eq: CKdV-tori-intro}, if $E_T$ and $F$ are both measurable sets with positive Lebesgue measure, can one achieve local mass-conserved controllability as established in \cite{Russell-Zhang-96}? There are two main difficulties for this problem:
\begin{enumerate}
    \item On the one hand, even in a highly simplified case such as $E_T\times F$, where $E_T$ is a time measurable set with $|E_T|>0$ and $F$  an generic open set, the known techniques do not lead to controllability and corresponding observability results for dispersive models.  
    \item On the other hand, an additional difficulty arises due to the mass-conserved constraint. So the new approach developed for the measurable setting must be adapted accordingly. 
\end{enumerate}

\subsection{Main result}
In this article we prove the following result on equation \eqref{eq: CKdV-tori-intro}.
\begin{theorem}\label{thm: CKDV-N-intro}
Let $T>0$. For any $M\in \mathbb{R}$, there exists a constant $\delta>0$ such that for any $u_0,u_1\in L^2(\T)$ with $\Avg{u_0}{\T}=\Avg{u_1}{\T}= M$ and $\|u_0-\Avg{u_0}{\T}\|_{L^2(\T)}+\|u_1-\Avg{u_1}{\T}\|_{L^2(\T)}<\delta$, one can find a control $h\in L^2([0,T]\times\T)$ such that the KdV equation
\begin{equation*}
\partial_tu+\partial_x^3u+u\partial_xu=\mathcal{L}(h)\mathbf{1}_{E_T\times F},
\end{equation*}
has a unique solution $u\in C([0,T];L^2(\T))$ satisfying that
\begin{equation*}
u(0,x)=u_0(x),u(T,x)=u_1(x).
\end{equation*}
\end{theorem}

Theorem \ref{thm: CKDV-N-intro} extends the results of \cite{Russell-Zhang-96} to the general space-time measurable setting. On the one hand, the case considered here corresponds to a ``rough control" situation, where $\supp g$ is merely a measurable subset of $\T$ with positive Lebesgue measure. On the other hand, the control is not applied over the entire time interval in the usual sense; instead, the time domain of actuation is itself merely a positively measurable subset in $[0,T]$.

\begin{remark}
In the proof, rather than employing the classical moment method, we adopt a strategy inspired by Miheev's theorem in harmonic analysis. By integrating this with a high/low frequency decomposition approach from the compactness-uniqueness method, we develop a novel three-step approach (we refer to Section \ref{sec: strategy of the proof} for more details) to obtain the local controllability result.

This new approach is particularly well-suited to the ``rough control" setting in $\T$. We employ it to establish observability results for a general class of dispersive operators $P(D)$, as stated in Theorem \ref{thm-spacetime-a} (see Section \ref{sec: Observability from space-time measurable sets}). Furthermore, we apply the method in a more specific context involving the KdV equation with mass-conservation constraints, as discussed in Section \ref{sec: linearized system}. We demonstrate that, in various settings, our method yields effective results and can be viewed as an improvement over the classical moment method and compactness-uniqueness method. For further discussion, we refer the reader to Section \ref{sec: new method}.
\end{remark}

\subsection{Strategy of the proof}\label{sec: strategy of the proof}
Our proof of Theorem \ref{thm: CKDV-N-intro} is divided into three steps. Firstly, in Section \ref{sec: Observability from space-time measurable sets}, we develop a new approach to establish the observability for a generalized dispersive equation from a measurable set with positive measure; see Theorem \ref{thm-spacetime-a}. Secondly, in Section \ref{sec: linearized system}, we adapt this new approach to prove a linear mass-conserved KdV observability and construct the control operator $\mathcal{K}$ based on the preceding observability. Thirdly, we use a fixed point argument in the Bourgain spaces to complete the proof for the nonlinear case in Section \ref{sec: nonlinear case}.

\subsubsection{Step 1: observability from measurable sets}
We first consider the following observability problem for a generalized dispersive model. Let $p: \mathbb{Z} \rightarrow \mathbb{R}$ be a monic polynomial of degree $d\geq 2$. We will use $\Z[x]$ to denote the ring of polynomials with integer coefficients. Consider the dispersive equation on $\mathbb{T}$
\begin{align}\label{equ-1}
 \partial_t u=i P(D) u, \quad u(0, x)=u_0(x) \in L^2(\mathbb{T})
\end{align}
where $D=i^{-1} \partial_x$ and $P(D)$ denotes the differential operator with symbol of $p(k)$. {From now on, unless otherwise specified, $p$ should satisfy the assumptions above.} We obtain
\begin{theorem}[Observability from space-time positive measure set]\label{thm-spacetime-a}
Given $T>0$,  let $a\in L^1_x\left(\T;L^\infty_t([0,T])\right)$. Then there exists a constant $C>0$ such that
\begin{align}\label{equ-55-14}
 \int_0^T\int_\T |a(t,x)e^{itP(D)}u_0|^2\d x \d t\leq C\|u_0\|^2_{L^2(\T)}, \quad \forall u_0\in L^2(\T).   
\end{align}
If, in addition, $\|a\|_{L^1_x(\T;L^\infty_t[0,T])}>0$, then there exists $C'>0$ such that
\begin{align}\label{equ-55-15-intro}
\|u_0\|^2_{L^2(\T)}\leq  C'\int_0^T\int_\T |a(t,x)e^{itP(D)}u_0|^2\d x \d t, \quad \forall u_0\in L^2(\T). 
\end{align}
\end{theorem}
\begin{remark}
We would like to emphasize that the observable function $a$ in Theorem \ref{thm-spacetime-a} is time-dependent and merely measurable in the time variable. Similar types of observability have been extensively investigated for the case where $a$ is time independent. We refer to \cite{Rosier97,LRZ-2010} for $a\in C(\T)$, and \cite{BurqZworski2019,alphonse2025} for $a\in L^1(\T)$. 
\end{remark}
In particular, when $G\subset [0,T]\times \T$ has positive measure, by taking $a=\mathbf{1}_{G}$  we obtain 
\begin{proposition}\label{thm-1}
Let $G\subset [0,T]\times\T$ be a measurable set with positive measure. Then there exists a constant $C(G)>0$ such that for all solutions to \eqref{equ-1}
\begin{align}\label{equ-ob-intro}
\|u_0\|^2_{L^2(\T)}\leq C(G)\iint_G|u(t,x)|^2\d x \d t.
\end{align} 
\end{proposition}
Very recently, Burq and Zhu obtained the first observability result for Schr\"odinger equation from a space-time measurable set in $\T^m$ in \cite{burq}. We take a different proof strategy, which in particular applies to the KdV equation with the mass conservation constraint. Concerning the proof of Theorem \ref{thm-spacetime-a}, we refer the reader to the detailed presentation in Section \ref{sec: Observability from space-time measurable sets}. In our approach, inspired by  the proof of Miheev's Theorem (see \cite{BonamiUCP2006}), we adapt the classical high/low frequency framework (see \cite{Rosier97} for example) to the setting of space-time measurable observation sets. More precisely, we first derive high-frequency bounds for  \eqref{equ-1} inspired by the work of Zygmund \cite{Zygmund1972}: \footnote{ Throughout, we shall use $A\lesssim B$ to denote $A\leq CB$ for some independent of relevant parameters $C>0$, and $A\gg B$ if $A\geq CB$ for a sufficiently large constant  $C>0$.}
\begin{gather}\label{eq: highest-intro}
\sum_{|k|>N}|\widehat{\varphi}(k)|^2\lesssim\iint_G\Big|\sum_{|k|>N}\widehat{\varphi}(k)e^{i(kx+p(k)t)}\Big|^2\d x \d t, \quad \forall \varphi\in L^2(\T).
    \end{gather}

    In the analysis of the low-frequency regime, rather than following the classical compactness-uniqueness method, we propose a novel method that enables the construction of a finite iterative scheme. This scheme successively incorporates low-frequency modes into the high-frequency estimate, and it mainly relies on the following two ingredients:

1. The observability \eqref{eq: highest-intro} holds uniformly if we replace $G$ by $G\cap (G-h)$ with $h$ sufficiently small.

2.  Establish an {\it augmented} observability inequality. Namely, for any $k_0<N$, we obtain
    \begin{equation}\label{eq: augment-est-intro}
    \sum_{|k|>N, k= k_0}|\widehat{\varphi}(k)|^2\leq C'(G)\iint_{G}\Big|\sum_{|k|>N, k= k_0}\widehat{\varphi}(k)e^{i(kx+p(k)t)}\Big|^2\d x \d t.
    \end{equation}
\begin{remark}
Based on the new observability result established in Theorem \ref{thm-spacetime-a}, we prove a uniformly exponential decay of the $L^2(\mathbb{T})$ norm for solutions to a general dispersive equation with spacetime damping. Specifically, we demonstrate that exponential decay occurs if the spacetime damping coefficient $a(t,x)$ is uniformly time-block precompact and bounded below by a positive constant; see Theorem \ref{thm-decay}. In particular, if $a(t,x)$ is periodic in $t$ (with period, say, $T$) and attains a positive lower bound on a subset of $[0,T] \times \mathbb{T}$ with positive measure, then an exponential decay of the solutions holds.

\end{remark}
\subsubsection{Step 2: Construct the control operator for the linearized KdV equation}
Employing the ideas of Hilbert uniqueness method, we construct the control operator for the linearized KdV equation $\partial_tu+\partial_x^3u=\mathcal{L}(h)\mathbf{1}_{E_T\times F}$ based on the desired observability in the form:
\begin{equation}\label{eq: mass-kdv-ob-intro}
\|\varphi\|^2_{L^2(\T)}\lesssim \iint_{E_T\times F}|\mathcal{L}(e^{-t\partial_x^3}\varphi)|^2\d x\d t.
\end{equation}
Due to the presence of the operator $\mathcal{L}$ on the right-hand side, the term $\mathcal{L}(e^{-t\partial_x^3}\varphi)$ alters the frequency components of the solution $e^{-t\partial_x^3}\varphi$. This distinguishes the mass-conserved linearized KdV equation from the general dispersive models considered previously. We adapt the preceding method to this mass-conserved setting and obtain the observability \eqref{eq: mass-kdv-ob-intro}.

In the same spirit, we first derive a high-frequency estimate as follows:
\begin{equation}\label{eq: kdv-high-est-intro}
\sum_{|k|>N} |\widehat{\varphi}(k)|^2
 \lesssim\int_{\T}\int_{E_T}|\sum_{|k|>N}\sum_{l\in\Z} e^{i(k^3,  l)(t, x)} L(k, l)\widehat{\varphi}(k)|^2\d t\d x,    
\end{equation}
where $L(k,l)$ is the coefficients of the matrix representation of the linear operator $\mathcal{L}$ under the $L^2$-basis $\{e^{ikx}\}_{k\in\Z}$ (see Lemma \ref{lem: 0-avg-op}). Then we show that \eqref{eq: kdv-high-est-intro} holds uniformly when replacing $E_T\cap (E_T-h)$. To add the low-frequency part successively, it suffices to prove the augmented observability: for $k_0\in [-N,N]$, 
\begin{align*}
\sum_{|k|>N, k=k_0 } |\widehat{\varphi}(k)|^2
 \leq C'(E_T,F) \int_{\T}\int_{E_T}|\sum_{|k|>N, k=k_0}\sum_{l\in\Z}e^{i(k^3,  l)(t, x)} L(k, l)\widehat{\varphi}(k)|^2\d t\d x.
\end{align*}
We conclude \eqref{eq: mass-kdv-ob-intro} based on a finite induction process. 

\begin{figure}[htp] 
    \centering
    \includegraphics[width=0.8\linewidth]{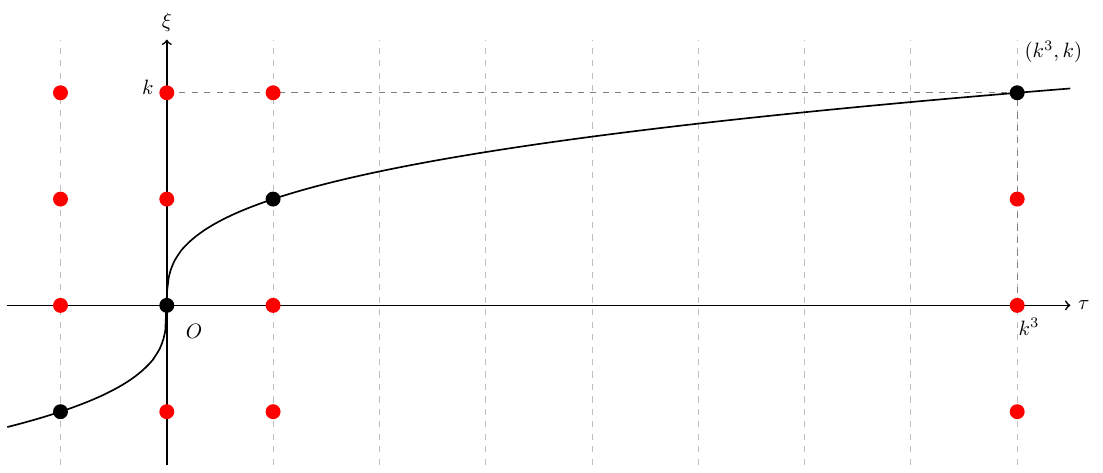}
    \caption{The frequency analysis for KdV control. In the classical linear KdV equation, only the frequencies $\{(k^3,k):k\in \mathbb{Z}\}$ play a role (see the black points). However, for the KdV equation with the mass conservation constraint, we need to analyze the frequencies $\{(k^3,l):k,l\in \mathbb{Z}\}$. As illustrated in the figure, infinitely many new frequencies come into play for each fixed $k$ (see the red points).}
    \label{fig:frequ} 
\end{figure}

Compared to the general approach presented in Section~\ref{sec: Observability from space-time measurable sets}, the mass-conserved KdV case presents additional features. As noted earlier, the mass conservation constraint alters the frequency structure of the adjoint solution. Consequently, the frequency of $\mathcal{L}(e^{-t\partial_x^3}\varphi)$ is no longer concentrated at the integer points on the curve $\tau=\xi^3$, but instead spreads across the entire frequency domain in $\xi$-direction, see Fig \ref{fig:frequ}. This new phenomenon requires a careful analysis of the operator $\mathcal{L}$, which is carried out in Section~\ref{sec: property of L}.  

\subsubsection{Step 3: nonlinear case}
This step is based on  a standard fixed-point argument. For the self-containment, we include some basic properties of Bourgain spaces and useful estimates. Equipped with them, we finally prove the local exact controllability of the nonlinear mass-conserved KdV equation.

\subsection{Novel observability inequality: beyond classical methods}\label{sec: new method}
In this subsection, we compare our method with the classical methods, in particular the moment method and the compactness-uniqueness method. 
\subsubsection{Beyond the moment method}
For interested readers, we give a brief review of the moment method in Appendix \ref{sec: moment method} (see also  \cite{FR71,TT-2007,lissy-2015}). The central aspect is the construction of a bi-orthonormal family $\{\phi_k\}$ to $\{e^{ik^3t}\}$ satisfying  $\int_0^T\phi_ke^{-il^3t}\d t=\delta_{kl}$. The classical construction relies on the Paley–Wiener Theorem. In analogue, in our setting, one needs a bi-orthonormal family  $\{\phi_k\}_k$ to $\{e^{i l^3t}\}_l$ satisfying the orthogonality condition 
\[
\int_{E_T}\phi_ke^{-il^3t}\d t=\delta_{kl},
\]
where $E_T$ is a positively measurable set. However, under this circumstance, the classical construction does not work.  

In fact, the moment method is a very powerful tool to ensure the controllability of many PDE models (Schr\"odinger \cite{BL10}, wave \cite{Russell-67}, heat \cite{FR71}). It is likely that our method could apply to these settings and one can obtain a more general result with control on a space-time measurable set.

\subsubsection{Beyond the compactness-uniqueness method}
For self-containment, we include a brief review on the classical compactness-uniqueness method in Appendix \ref{sec: compant-unique method}, see also \cite{sun-kp,RS-kp2}. Classically, we obtain the observability inequality via the Ingham inequality or microlocal methods. In the application of Ingham's inequality, replacing the time interval $[0,T]$ by a time measurable set is a hard task. As for the microlocal method, it is based on symbolic calculus, and its application usually requires sufficient regularity assumptions (for instance, $C^2$).

To move beyond the compactness-uniqueness method, one can either construct a carefully designed approximation scheme and establish associated uniform estimates to pass the limit, or focus more directly on the analysis of trigonometric series. As previously mentioned, a notable recent development in this direction is due to Burq and Zhu, who introduced an elegant approximation scheme to establish the first observability result from measurable sets for the Schrödinger equation on tori \cite{burq}.  Alternatively, our approach, outlined in detail in Section \ref{sec: new method}, relies on a refined analysis of the trigonometric series that allows us to deal with the difficulty caused by the measurable sets.

\subsection*{Acknowledgements}
The authors would like to thank Nicolas Burq and Chenmin Sun for valuable and useful discussions.    Ming Wang was partially supported by the National Natural Science Foundation of China under grants 12171442 and 12171178. Shengquan Xiang is partially supported by NSF of China 12301562, and by “The Fundamental Research Funds for the Central Universities, 7100604200, Peking University”. Jingrui Niu is supported by Defi Inria EQIP.

\section{Observability from space-time measurable sets}\label{sec: Observability from space-time measurable sets}
In this section, we study the general dispersive equation 
\[
\partial_t u=i P(D) u, \quad u(0, x)=u_0(x),
\]
where $\widehat{P(D) u}(k):=p(k) \widehat{u}(k)$, $\forall k \in \mathbb{Z}$.
Here $\widehat{u}(k)$ stands for the $k$-th Fourier coefficient of $u$ defined as
$
\widehat{u}(k)=\frac{1}{2 \pi} \int_0^{2 \pi} u(x) e^{-i k x} \mathrm{~d} x.
$
We deal with the observability inequalities associated with the general dispersive equation \eqref{equ-1}, particularly focusing on the space-time measurable set $G\subset[0,T]\times\T$ with positive measure. The primary goal of this section is to prove Theorem \ref{thm-spacetime-a}, which is structured in the following four steps.

$\bullet$\textbf{ Step 1: }Strichartz estimates. We establish  in subsection \ref{sec: Strichartz} that
\begin{gather*}
\|e^{itP(D)}u_0\|_{L^\infty_x(\mathbb{T}; L^2_t([0,T]))} \lesssim \|u_0\|_{L^2(\mathbb{T})},\quad 
\|e^{itP(D)}u_0\|_{L^4(\T^2)}\lesssim \|u_0\|_{L^2(\T)}.
\end{gather*}
We begin by proving the first estimate in Lemma \ref{lem-Stri-LinftyL2}, which corresponds to the establishment of \eqref{equ-55-14}. The latter estimate, in turn, serves as a preliminary step toward the development of the low-frequency analysis in Step 3. 

$\bullet$\textbf{ Step 2: }High-frequency estimates. We prove a high-frequency estimate in Lemma \ref{lem-high-freq} of subsection \ref{sec: high-fre-est}:  
\begin{equation*}
\frac{|G|}{2}\sum_{|k|>N, k\in \Z}|a_k|^2\leq \iint_G\Big|\sum_{|k|>N, k\in \Z}a_ke^{i(kx+p(k)t)}\Big|^2\d x \d t, \quad \forall \{a_k\}\in l^2.
\end{equation*} 

$\bullet$\textbf{ Step 3: }Low-frequency analysis. This step represents the core of our new strategy. Rather than relying on the classical unique continuation property, we carry out the low-frequency analysis by exploiting the stability of the high-frequency estimate under perturbations of the measurable observed set. More precisely, we prove that for a fixed high-frequency set $\Lambda=\{(k,p(k)):|k|>N\}\subset\Z^2$, we have
\begin{equation}\label{eq: stable-high-est}
\sum_{(k,p(k))\in\Lambda}|a_k|^2\leq C(G)\iint_{G\cap (G-h)}\Big|\sum_{(k,p(k))\in\Lambda}a_ke^{i(k x+p(k)t)}\Big|^2\d x \d t, \quad \forall \{a_k\}\in l^2,
\end{equation}
which implies that the high-frequency estimate remains stable under small translations of the set $G$, where $G-h:=\{z-h: z=(x,t)\in G\}$ for $|h|$ sufficiently small.

Under the condition \eqref{eq: stable-high-est}, we can incorporate a low-frequency point $\lambda\notin\Lambda$ such that the following augmented frequency observability holds  
\begin{equation*}
\sum_{(k,p(k))\in\Lambda\cup\{\lambda\}}|a_k|^2\leq C'(G)\iint_{G}\Big|\sum_{(k,p(k))\in\Lambda\cup\{\lambda\}}a_ke^{i(kx+p(k)t)}\Big|^2\d x \d t, \quad \forall \{a_k\}\in l^2,
\end{equation*}
By repeating the translation uniform argument for high-frequency estimates, we deduce that inequality \eqref{eq: stable-high-est} also holds for the updated set $\Lambda\cup\{\lambda\}$. This allows us to add another low-frequency point. By iterating this procedure a finite number of times, since the number of low-frequency points is finite, we complete the proof of Proposition \ref{thm-1}.

$\bullet$\textbf{ Step 4: } From the indicator function to the general case. We prove that for any $a\in L^{1}_x(\T;L^{\infty}_t([0,T]))$
\[
\|u_0\|^2_{L^2(\T)}\leq  C'\int_0^T\int_\T |a(t,x)e^{itP(D)}u_0|^2\d x \d t, \quad \forall u_0\in L^2(\T),  
\]
by approximation and contradiction arguments.

\subsection{Strichartz estimates}\label{sec: Strichartz}
In this part we prove two estimates; see Lemma \ref{lem-Stri-LinftyL2}--\ref{lem-str}.
\begin{lemma}\label{lem-Stri-LinftyL2}
Let $T>0$. Then, there exists a constant $C_{\rm{Str}}=C_{\rm{Str}}(d,T)>0$ such that
\begin{align}\label{equ-55-3}
\|e^{itP(D)}u_0\|_{L_x^\infty(\T;L^2_t([0,T]))} \leq C_{\rm{Str}} \|u_0\|_{L^2(\T)}, \quad \forall u_0\in L^2(\T).
\end{align}
\end{lemma}
\begin{proof}
We expand the initial state $u_0$ into the Fourier series, i.e., $u_0=\sum_{k\in \Z}c_ke^{ikx}$. Then, the solution has the form $e^{itP(D)}u_0=\sum_{k\in \Z}c_ke^{itp(k)+ikx}$ and the left-hand side of \eqref{equ-55-3} reads as
\begin{align}\label{equ-55-4}
\|e^{itP(D)}u_0\|^2_{L_x^\infty(\T;L^2_t([0,T]))}    
=\sup_{x\in \T}\sum_{k,l\in \Z}\int_0^Tc_k\overline{c_l}e^{it(p(k)-p(l))}e^{ix(k-l)}\d t. 
\end{align}
To bound \eqref{equ-55-4}, we split the sum into two parts
\begin{gather*}
I_1:=\sum_{k,l\in \Z, p(k)=p(l)}\int_0^Tc_k\overline{c_l}e^{ix(k-l)}\d t,\;\;
I_2:=\sum_{k,l,p(k)\neq p(l)}\int_0^Tc_k\overline{c_l}e^{it(p(k)-p(l))}e^{ix(k-l)}\d t.
\end{gather*}
The contribution of the first part is bounded by
\begin{align*}
|I_1|\leq \sup_{x\in \T}\sum_{k,l\in \Z, p(k)=p(l)}|c_k\overline{c_l}|\d t\leq T\sum_{k,l\in \Z}A_{k,l}|c_k\overline{c_l}|,
\end{align*}
where $A_{k,l}=1_{\{(k,l)\in \Z^2: p(k)=p(l)\}}(k,l)$. By the algebra fundamental theorem, for every $k\in \Z$, there are at most $d$ elements $l$ such that $p(l)=p(k)$. Thus,
we have
\begin{align}\label{equ-55-6}
\sup_{k\in \Z}\sum_{l\in \Z}A^2_{k,l}\leq d, \quad \sup_{l\in \Z}\sum_{k\in \Z}A^2_{k,l}\leq d.
\end{align}
Using \eqref{equ-55-6} and the Cauchy–Schwarz inequality, we get
\begin{align}\label{equ-55-7}
|I_1| \leq \frac{T}{2}\left( \sum_{k,l\in\Z} |c_k|^2 A_{k,l}+\sum_{k,l\in\Z} |c_l|^2 A_{k,l}\right)\leq dT\sum_{k\in\Z}|c_k|^2.
\end{align}

The contribution of the second part is bounded by
\begin{align*}
|I_2|\leq \left|\sup_{x\in \T}\sum_{k,l\in \Z,p(k)\neq p(l)} c_k\overline{c_l} e^{ix(k-l)} \frac{e^{iT(p(k)-p(l))}-1}{i(p(k)-p(l))}  \right| \lesssim \sum_{k,l} \frac{|c_k\overline{c_l}|}{\langle p(k)-p(l)\rangle},  
\end{align*}
where we used the notation $\langle x\rangle = 1+|x|$ and the lower bound $|p(k)-p(l)|\geq 1$ if $p(k)\neq p(l)$. 
The latter follows from the fact that $p$ only takes values in $\Z$. We claim that
\begin{align}\label{equ-55-9} 
\sum_{k,l} \frac{|c_k\overline{c_l}|}{\langle p(k)-p(l)\rangle}\lesssim \sum_{k\in \Z}|c_k|^2.    
\end{align}
Once this is done, $|I_2|\lesssim \sum_{k\in \Z}|c_k|^2$. Combining with \eqref{equ-55-7}, we complete the proof of \eqref{equ-55-3}. Hence, it remains to prove \eqref{equ-55-9}. To this end, we split the sum into four terms as
$$
\sum_{k,l} \frac{|c_k\overline{c_l}|}{\langle p(k)-p(l)\rangle}
=(\sum_{|k|,|l|\lesssim 1}+\sum_{|k| \lesssim 1, |l|\gg 1}+\sum_{|k|\gg 1,|l|\lesssim 1} +\sum_{|k|,|l|\gg 1}) \frac{|c_k\overline{c_l}|}{\langle p(k)-p(l)\rangle}. 
$$
We estimate term by term. For the first term, we have
\begin{align}\label{equ-55-10} 
 \sum_{|k|,|l|\lesssim 1}  \frac{|c_k\overline{c_l}|}{\langle p(k)-p(l)\rangle} \leq \sum_{|k|,|l|\lesssim 1}   |c_k\overline{c_l}|\lesssim (\sum_{|k|,|l|\lesssim 1}   |c_k\overline{c_l}|^2 )^{1/2}=\sum_{k\in \Z}|c_k|^2.
\end{align}
For the second term, if $|k| \lesssim 1, |l|\gg 1$, then the leading term of $|p(k)-p(l)|$ is $|l|^d$, thus by Cauchy–Schwarz inequality,
\begin{align}\label{equ-55-11} 
 \sum_{|k|\lesssim 1,|l|\gg1}  \frac{|c_k\overline{c_l}|}{\langle p(k)-p(l)\rangle} &\leq \sum_{|k|\lesssim 1,|l|\gg 1}   \frac{|c_k\overline{c_l}|}{(1+|l|^d) }\leq \sum_{|k|\lesssim 1}|c_k|(\sum_{|l|\gg 1}\frac{1}{(1+|l|^d)^2})^{1/2}(\sum_{|l|\gg 1}|c_l|^2)^{1/2}\nonumber\\
 &\lesssim \sum_{|k|\lesssim 1}|c_k|(\sum_{|l|\gg 1}|c_l|^2)^{1/2}\lesssim \sum_{k\in \Z}|c_k|^2.
\end{align}
Similarly, for the third term,
\begin{align}\label{equ-55-12} 
\sum_{|k|\gg 1,|l|\lesssim 1}  \frac{|c_k\overline{c_l}|}{\langle p(k)-p(l)\rangle} \lesssim \sum_{k\in \Z}|c_k|^2.    
\end{align}
For the last term, we claim that
\begin{align}\label{equ-55-13} 
\sum_{|k|,|l|\gg 1}  \frac{|c_k\overline{c_l}|}{\langle p(k)-p(l)\rangle} \lesssim \sum_{k\in \Z}|c_k|^2.    
\end{align}
Combining the bounds \eqref{equ-55-10}-\eqref{equ-55-13}, we conclude \eqref{equ-55-9}.

It remains to show the claim \eqref{equ-55-13}. First, we consider the case when the degree $d$ of $p$ is odd. Split further the sum in \eqref{equ-55-13} as two terms
$$
\sum_{|k|,|l|\gg 1}  \frac{|c_k\overline{c_l}|}{\langle p(k)-p(l)\rangle}=
\sum_{|k|,|l|\gg 1, k=l} \frac{|c_k\overline{c_l}|}{\langle p(k)-p(l)\rangle}+ \sum_{|k|,|l|\gg 1, k\neq l} \frac{|c_k\overline{c_l}|}{\langle p(k)-p(l)\rangle}.
$$
The first term is clearly bounded by $\sum_{k\in \Z}|c_k|^2$. For the second term, if $d$ is odd and $k\neq l$, then the leading term of $\langle p(k)-p(l)\rangle$ is $|l-k|(l^{d-1}+k^{d-1})$, thus
\begin{align*}
\sum_{|k|,|l|\gg 1,l\neq k}  \frac{|c_k\overline{c_l}|}{\langle p(k)-p(l)\rangle}\lesssim    \sum_{|k|,|l|\gg 1,l\neq k}  \frac{|c_k\overline{c_l}|}{|l-k|(l^{d-1}+k^{d-1})}\lesssim  \sum_{k\in \Z}|c_k|^2,
\end{align*}
where in the last step we used Cauchy–Schwarz and $\sum_{|k|,|l|\gg 1,k\neq l}\frac{1}{|l-k|^2(l^{d-1}+k^{d-1})^2}\lesssim 1$.
This proves \eqref{equ-55-13} if $d$ is odd.

Now, we consider the other case when $d$ is even. Split further the sum in \eqref{equ-55-13} as two terms
$$
\sum_{|k|,|l|\gg 1}  \frac{|c_k\overline{c_l}|}{\langle p(k)-p(l)\rangle}=
\sum_{|k|,|l|\gg 1, l^2=k^2}\frac{|c_k\overline{c_l}|}{\langle p(k)-p(l)\rangle}+ \sum_{|k|,|l|\gg 1, l^2\neq k^2}\frac{|c_k\overline{c_l}|}{\langle p(k)-p(l)\rangle}.
$$
For the first term, we have
\begin{align*}
\sum_{|k|\gg 1,k^2=l^2}  \frac{|c_k\overline{c_l}|}{\langle p(k)-p(l)\rangle} \leq \sum_{|k| \gg 1,l=\pm k}   |c_k\overline{c_l}|   \lesssim  \sum_{k\in \Z}|c_k|^2.  
\end{align*}
For the second term, if $d$ is even and $k^2\neq l^2, |k|,|l|\gg 1$, then
$$
|k^d-l^d|\gtrsim |l^2-k^2|(l^{d-2}+k^{d-2}),
$$
which is the leading term of $\langle p(k)-p(l)\rangle$. Thus
\begin{align*}
\sum_{|k|\gg 1,k^2\neq l^2}  \frac{|c_k\overline{c_l}|}{\langle p(k)-p(l)\rangle} \leq \sum_{|k|,|l| \gg 1,k^2\neq l^2} \frac{|c_k\overline{c_l}|}{|l^2-k^2|(l^{d-2}+k^{d-2})}     \lesssim  \sum_{k\in \Z}|c_k|^2,  
\end{align*}
where we used the fact
$$
\sum_{|k|,|l| \gg 1,k^2\neq l^2} \frac{1}{|l^2-k^2|^2(l^{d-2}+k^{d-2})^2}\leq \sum_{|k|,|l| \gg 1,k^2\neq l^2} \frac{1}{|l^2-k^2|^2 }\lesssim 1.
$$
So \eqref{equ-55-13} also holds if $d$ is even. This completes the proof.
\end{proof}
The next lemma presents the proof of $\|e^{itP(D)}u_0\|_{L^4(\T^2)}\lesssim \|u_0\|_{L^2(\T)}$. The proof exploits the orthogonality of triangle polynomials and some arithmetic properties of the polynomial $p$.
\begin{lemma}\label{lem-str}
There is a constant $C_{\rm{Zyg}}=C_{\rm{Zyg}}(d)>0$ such that
\begin{align}\label{equ-G-2}
\left\|\sum_{k\in \Z}a_ke^{i(kx+p(k)t)}\right\|_{L^4(\T^2)}\leq C_{\rm{Zyg}}\left( \sum_{k\in \Z}|a_k|^2\right)^{1/2}.
\end{align}
\end{lemma}

\begin{proof}
Let $\lambda_k=(k,p(k))$, and $f=\sum_{k\in \Z}a_ke^{i\langle \lambda_k, (x,t)\rangle}$. Here $\langle \cdot, \rangle$ denotes the scalar product in $\R^2$. Then, we have
\begin{align}\label{equ-G-3}
 f\overline{f}=\sum_{k\in \Z}|a_k|^2+\sum_{k_1\neq k_2}a_{k_1}\overline{a}_{k_2}e^{i\langle \lambda_{k_1}-\lambda_{k_2}, (x,t)\rangle}.
\end{align}
Since $p\in \Z[x]$, we deduce that $\lambda_k\in \Z^2$. Hence, taking $L^2(\T^2)$ on both sides of \eqref{equ-G-3} and using the orthogonality of $e^{i\langle \lambda_k, (x,t)\rangle}$ in $L^2(\T^2)$, we obtain
\begin{align}\label{equ-G-4}
\|f\overline{f}\|_{L^2(\T^2)}\leq 2\pi\sum_{k\in \Z}|a_k|^2+\left(\Theta \sum_{k_1\neq k_2}|a_{k_1}\overline{a}_{k_2}|^2\right)^{1/2},
\end{align}
where $\Theta$ is a quantity defined by
 \begin{align}\label{equ-G-5}
\Theta=\sup_{\alpha\in \Z^2\backslash \{0\}}\# \Big\{(k_1,k_2)\in \Z^2: \lambda_{k_1}-\lambda_{k_2}=\alpha\Big\}.
\end{align}
Since $\|f\|^2_{L^4(\T^2)}=\|f\overline{f}\|_{L^2(\T^2)}$, using the Cauchy–Schwarz inequality, we deduce from \eqref{equ-G-4} that
$$
\|f\|^2_{L^4(\T^2)}\leq (2\pi +\Theta^{1/2})\sum_{k\in \Z}|a_k|^2.
$$
Thus the Strichartz estimate \eqref{equ-G-2} follows if one can show 
\begin{align}\label{equ-G-00}
    \Theta\leq d-1.
\end{align}

To see this, arbitrarily fix $\alpha=(\alpha_1,\alpha_2)\in \Z^2\backslash \{0\}$. By \eqref{equ-G-5}, $\Theta$ is the number of solutions $(k_1,k_2)\in \Z^2$ to the equation   $\lambda_{k_1}-\lambda_{k_2}=\alpha$, which is equivalent to
\begin{align}
k_1 - k_2 = \alpha_1,  \label{equ-G-6}\\
p(k_1) - p(k_2) = \alpha_2.  \label{equ-G-7}
\end{align}
Let $p(k)=k^d+\mbox{l.o.t}$ where $\mbox{l.o.t}$ denotes the lower order terms. Then
\begin{align}\label{equ-G-8}
   p(k_1) - p(k_2)=(k_1-k_2)(k_1^{d-1}+k_1^{d-2}k_2+\cdots+k_2^{d-1})+\mbox{l.o.t}.
\end{align}
Substituting $k_2=k_1-\alpha_1$ (by \eqref{equ-G-6}) into \eqref{equ-G-8}, we see $p(k_1) - p(k_2)$ is a polynomials of $k_1$ with degree $d-1$. By the fundamental theorem of algebra, we find that \eqref{equ-G-6}-\eqref{equ-G-7} has at most $d-1$ solutions. This implies that $\Theta\leq d-1$. Thus \eqref{equ-G-00} holds. Consequently, we conclude \eqref{equ-G-2}
\end{proof}

\subsection{Observability from spacetime measurable sets}\label{sec: general ob}
This section is devoted to the proof of  Proposition \ref{thm-1}. As usual, we expand the initial state $u_0$ and the solution $u=e^{itP(D)}u_0$ into Fourier series:
\begin{equation*}
u_0(x)=\sum_{k\in \Z}a_ke^{ikx},\;\;u(t,x)=\sum_{k\in \Z}a_ke^{i(kx+p(k)t)}.
\end{equation*}
So the desired observability is equivalent to
\begin{align}\label{equ-G-1}
 \sum_{k\in \Z}|a_k|^2\leq C(G)\iint_G|\sum_{k\in \Z}a_ke^{i(kx+p(k)t)}|^2\d x \d t, \quad \forall \{a_k\}\in l^2.
\end{align}
\subsubsection{High-frequency estimates}\label{sec: high-fre-est}
In this part, inspired by the classical work of Zygmund \cite{Zygmund1972}, we split the observability into the high-frequency part and the low-frequency part, and then deal with them respectively. {For simplicity, we reduce the analysis of the space-time measurable set $G\subset[0,T]\times\T$ to $G\subset\T^2$.} The high-frequency estimate reads as follows.
\begin{lemma}\label{lem-high-freq}
If $G\subset\T^2$ has positive measure,  then there exists a constant $N>0$ depending only on $G$ and $p$ such that
\begin{align}\label{equ-G-10}
\frac{|G|}{2}\sum_{|k|>N, k\in \Z}|a_k|^2\leq \iint_G\Big|\sum_{|k|>N, k\in \Z}a_ke^{i(kx+p(k)t)}\Big|^2\d x \d t, \quad \forall \{a_k\}\in l^2.
\end{align}
\end{lemma}
\begin{proof}
Note that the same relation as \eqref{equ-G-3} holds for the sum over $|k|\geq N$. Integrating over $G$, we obtain
\begin{multline}\label{equ-G-11}
  \iint_G|\sum_{|k|>N, k\in \Z}a_ke^{i(kx+p(k)t)}|^2\d x \d t\\
  =
|G|\sum_{|k|>N, k\in \Z}|a_k|^2+\iint_G \sum_{k_1\neq k_2, |k_1|,|k_2|>N}a_{k_1}\overline{a}_{k_2}e^{i\langle \lambda_{k_1}-\lambda_{k_2}, (x,t)\rangle}\d x \d t.
\end{multline}
The last term on the RHS of \eqref{equ-G-11} can be written as
\begin{align}\label{equ-G-12}
    \sum_{k_1\neq k_2, |k_1|,|k_2|>N}a_{n_1}\overline{a}_{n_2} \widehat{\mathbf{1}_G}(\lambda_{k_2}-\lambda_{k_1}),
\end{align}
where $\mathbf{1}_G$ denotes the characteristic function of $G$, and $\widehat{\cdot}$ denotes the space-time Fourier transform.
Using the Cauchy–Schwarz inequality, we have
\begin{align}\label{equ-G-13}
|\eqref{equ-G-12}|&\leq \left(\sum_{k_1\neq k_2, |k_1|,|k_2|>N}|a_{k_1}\overline{a}_{k_2}|^2\right)^{1/2}\left(\sum_{k_1\neq k_2, |k_1|,|k_2|>N}|\widehat{\mathbf{1}_G}(\lambda_{k_2}-\lambda_{k_1})|^2\right)^{1/2}\nonumber\\
&\leq \left(\sum_{k_1\neq k_2, |k_1|,|k_2|>N}|a_{k_1}\overline{a}_{k_2}|^2\right)^{1/2}\left(\Theta\sum_{|\alpha|>N, \alpha\in \Z^2}|\widehat{\mathbf{1}_G}(\alpha)|^2\right)^{1/2},
\end{align}
where $\Theta$ is the same as in \eqref{equ-G-5}, and we have used the fact
\begin{align}\label{equ-G-14}
    |\lambda_{k_2}-\lambda_{k_1}|= (|k_2-k_1|^2+|p(k_2)-p(k_1)|^2)^{1/2}>N, \quad k_1\neq k_2, |k_1|,|k_2|>N,
\end{align}
for $N>0$ large enough. Indeed, we note that
\begin{align*}
 |p(k_1)-p(k_2)|\gtrsim  \begin{cases}
 |k_1-k_2|(k_1^{d-1}+k_2^{d-1}), \quad d \mbox{ odd },\\
 |k_1^2-k_2^2|(k_1^{d-2}+k_2^{d-2}), \quad d \mbox{ even }.
 \end{cases} 
\end{align*}
 This proves \eqref{equ-G-14} directly if $d$ is odd. If $d$ is even,  we have
$$
|\lambda_{k_2}-\lambda_{k_1}|\gtrsim |k_2-k_1|(1+|k_2+k_1|)\gtrsim N, \quad k_1\neq k_2, |k_1|,|k_2|>N.
$$
According to the proof of Lemma \ref{lem-str}, we know $\Theta\leq C(d)$. Thus, by the Cauchy–Schwarz inequality, we deduce from \eqref{equ-G-13} that
\begin{align}\label{equ-G-18}
|\eqref{equ-G-12}|\leq C^{1/2}(d)\sum_{|k|>N,k\in \Z}|a_k|^2\left(\sum_{|\alpha|>N,\alpha\in \Z^2}|\widehat{\mathbf{1}_G}(\alpha)|^2\right)^{1/2}.
\end{align}
Since $G\subset\T^2$, $\mathbf{1}_G\in L^2(\T^2)$, and by the Plancherel theorem, $\sum_{\alpha\in \Z^2}|\widehat{\mathbf{1}_G}(\alpha)|^2<\infty$. Thus, there exsits $N>0$ large enough such that
\begin{align}\label{equ-G-19}
C^{1/2}(d)\left(\sum_{|\alpha|>N,\alpha\in \Z^2}|\widehat{\mathbf{1}_G}(\alpha)|^2\right)^{1/2}<\frac{|G|}{2}.
\end{align}
Plugging \eqref{equ-G-18}-\eqref{equ-G-19} into \eqref{equ-G-11}, we obtain \eqref{equ-G-10}.
\end{proof}

Using the Strichartz estimate in Lemma \ref{lem-str}, we show a stablity of the observability in Lemma \ref{lem-high-freq} slightly, namely \eqref{equ-G-10} still holds if $G$ is replaced by a smaller set $G\cap (G-h)$ if $h$ is small enough. Recall that $G-h=\{z-h: z=(t,x)\in G\}$. 
\begin{lemma}\label{lem: perturbative}
Let $G$ and $N$ be the same as those in Lemma \ref{lem-high-freq}. Then there exists a constant $\delta_0>0$ depending only on $G$ and $p$ such that
$$
\frac{|G|}{4}\sum_{|k|>N, k\in \Z}|a_k|^2\leq \iint_{G\cap (G-h)}\Big|\sum_{|k|>N, k\in \Z}a_ke^{i(kx+p(k)t)}\Big|^2\d x \d t, \quad \forall \{a_k\}\in l^2
$$
for all $h\in \T^2, |h|<\delta_0$.
\end{lemma}
\begin{proof}
Note that $G$ equals to the union of $G\cap (G-h)$ and $G\backslash (G-h)$, thanks to the bound \eqref{equ-G-10}, it suffices to show that
\begin{align}
\label{equ-514-1}    
\iint_{G\backslash (G-h)}\Big|\sum_{|k|>N, k\in \Z}a_ke^{i(kx+p(k)t)}\Big|^2\d x \d t\leq\frac{|G|}{4}\sum_{|k|>N, k\in \Z}|a_k|^2, \quad \forall \{a_k\}\in l^2,
\end{align}
for all $h\in \T^2, |h|<\delta_0$.

Indeed, thanks to the Strichartz estimate in Lemma \ref{lem-str}, we have
\begin{align}\label{equ-514-2}  
\|\sum_{|k|>N, k\in \Z}a_ke^{i(kx+p(k)t)}\|_{L^4(\T^2)}\leq C\left(\sum_{|k|>N}|a_k|^2 \right)^{1/2},
\end{align}
with $C>0$ depending only on (the degree of) $p$. For every subset  $F\subset \T^2$, by H\"older inequality and \eqref{equ-514-2}, 
\begin{align}\label{equ-514-3} 
\|\sum_{|k|>N, k\in \Z}a_ke^{i(kx+p(k)t)}\|_{L^2(F)}\leq |F|^{\frac{1}{4}}\|\sum_{|k|>N, k\in \Z}a_ke^{i(kx+p(k)t)}\|_{L^4(F)}\nonumber
\\
\leq 
|F|^{\frac{1}{4}}\|\sum_{|k|>N, k\in \Z}a_ke^{i(kx+p(k)t)}\|_{L^4(\T^2)}
\leq C|F|^{\frac{1}{4}}   \left(\sum_{|k|>N}|a_k|^2 \right)^{1/2}.
\end{align}
For every $h\in \T^2$, letting $F=G\backslash (G-h)$ in \eqref{equ-514-3}, we have
$$
\iint_{G\backslash (G-h)}\Big|\sum_{|k|>N, k\in \Z}a_ke^{i(kx+p(k)t)}\Big|^2\d x \d t\leq
C^2|G\backslash(G-h)|^{\frac{1}{2}}\sum_{|k|>N}|a_k|^2.
$$
But the measure $|G\backslash(G-h)|\to 0$ as $h\to 0$, then we conclude \eqref{equ-514-1} if $h$ is small enough.
\end{proof}

\subsubsection{Low-frequency analysis}\label{sec: low-fre-est}
To deal with the low frequency part, we need the following lemma, which is a variant version of \cite[p.10, Lemma 2.13]{BonamiUCP2006}. Note that the following result holds for $\T^{m}(m\geq1)$, but we shall only use it in $\T^2$ in this paper.
\begin{lemma}[Augmented observability]\label{lem-add-freq}
Let $\Lambda$ be a subset of $\Z^m$ and let $F$ be a measure subset of $\T^m$ with positive measure. Assume that for  there exists a constant $C(F)>0$ such that
\begin{align}\label{equ-lem-add-freq-1}
\sum_{k\in\Lambda}|a_k|^2\leq C(F)\int_{F\cap (F-h)}|\sum_{k\in \Lambda}a_ke^{ik\cdot x}|^2\d x, \quad \forall \{a_k\}\in l^2,
\end{align}
holds for all $|h|\leq \delta_0$ with some $\delta_0>0$.
Let $\lambda\in \Z^m \backslash \Lambda$. Then
there exists a constant $C'(F)>0$ such that
\begin{align}\label{equ-lem-add-freq-2}
\sum_{k\in\Lambda\cup\{\lambda\}}|a_k|^2\leq C'(F)\int_F|\sum_{k\in \Lambda\cup \{\lambda\}}a_ke^{ik\cdot x}|^2\d x, \quad \forall \{a_k\}\in l^2.
\end{align}
\end{lemma}
\begin{proof}
We include the proof here for the reader's convenience. Fix a positive measure set $F\subset \T^m$. Without loss of generality, we assume $\lambda=0$. We argue by contradiction. Suppose that \eqref{equ-lem-add-freq-2} fails, then there exist a sequence $\{f_n\}_{n\in\N}\subset L^2(\T^m)$ with $\mathrm{supp}\widehat{f_n}\subset \Lambda$ and $c_n\in \C$ such that
\begin{align}\label{equ-lem-add-freq-3}
\|f_n+c_n\|^2_{L^2(\T^m)}=1, \quad \int_F|f_n+c_n|^2\d x \to 0 \text{ as } n\to \infty.  
\end{align}
Here and below, we say $\mathrm{supp}\widehat{f}\subset\Lambda$ if $f=\sum_{k\in \Lambda}a_ke^{ik\cdot x}$ for some $\{a_k\}\in l^2$. Since $0\notin \Lambda$, by the orthognolity, we have $\|f_n\|_{L^2(\T^m)}\leq 1$ and $|c_n|\leq 1$ for all $n\in\N$. It follows that there exist subsequces, still denoted by $f_n,c_n$, such that $f_n\rightharpoonup f$ weakly in $L^2(\T^m)$ and $c_n\to c$ for some $f\in L^2(\T^m)$ with $\mathrm{supp}\widehat{f}\subset\Lambda$ and $c\in \C$. Thus $f_n+c_n\rightharpoonup f+c$ weakly in $L^2(\T^m)$. 

We now prove that $f_n$ converges to $f$ in $L^2(\T^m)$ strongly. It is natural to show that $\{f_n\}_{n\in\N}$ is a Cauchy sequence in $L^2(\T^m)$. By the Plancherel theorem, using the fact that $\mathrm{supp }\widehat{f_n}\subset \Lambda$, for any $n\in\N$, we have
\begin{equation*}
\|f_l-f_n\|^2_{L^2(\T^m)}=\sum_{k\in\Lambda}|\widehat{f_l}(k)-\widehat{f_n}(k)|^2.
\end{equation*}
Due to \eqref{equ-lem-add-freq-1}, in particular, taking $h=0$, we know that
\begin{align*}
\sum_{k\in\Lambda}|\widehat{f_l}(k)-\widehat{f_n}(k)|^2&\leq C(F)\int_F|\sum_{k\in\Lambda}(\widehat{f_l}(k)-\widehat{f_n}(k))e^{ik\cdot x}|^2\d x\\
&=C(F)\int_F|\sum_{k\in\Lambda}(\widehat{f_l}(k)-\widehat{f_n}(k))e^{ik\cdot x}+(c_l-c_n)-(c_l-c_n)|^2\d x\\
&=C(F)\int_F|\sum_{k\in\Lambda\cup\{0\}}\widehat{f_l}(k)e^{ik\cdot x}-\sum_{k\in\Lambda\cup\{0\}}\widehat{f_n}(k)e^{ik\cdot x}-(c_l-c_n)|^2\d x\\
&\leq 3C(F)\int_F|\sum_{k\in\Lambda\cup\{0\}}\widehat{f_l}(k)e^{ik\cdot x}|^2\d x+3C(F)\int_F|\sum_{k\in\Lambda\cup\{0\}}\widehat{f_n}(k)e^{ik\cdot x}|^2\d x\\
&+3C(F)|F||c_l-c_n|^2.
\end{align*}
Here we adopt a convention that $\widehat{f_n}(0)=c_n$. According to \eqref{equ-lem-add-freq-3}, we know that the first and second terms appearing on the right-hand side tend to $0$, as $l,n\to\infty$. Therefore, for any $\varepsilon>0$, there exists $N_f\in\N$ sufficiently large such that
\[
3C(F)\int_F|\sum_{k\in\Lambda\cup\{0\}}\widehat{f_n}(k)e^{ik\cdot x}|^2\d x<\frac{\varepsilon}{3},\forall n> N_f.
\]
Since $\{c_n\}_{n\in\N}$ is a Cauchy sequence in $\C$, we know that there exists $N_c\in\N$ such that 
\[
3C(F)|F||c_l-c_n|^2<\frac{\varepsilon}{3},\forall l,n>N_c.
\]
As a consequence, as $l,n>\max\{N_c,N_f\}$, we have
\[
\sum_{k\in\Lambda}|\widehat{f_l}(k)-\widehat{f_n}(k)|^2<\varepsilon,\mbox{ i.e., }\{f_n\}\mbox{ is a Cauchy sequence in }L^2(\T^m). 
\]
Hence, $f_n\to f$ strongly in $L^2(\T^m)$. Combining with \eqref{equ-lem-add-freq-3}, we have
\[
\|f+c\|^2_{L^2(\T^m)}=\lim_{n\to\infty}\|f_n+c_n\|^2_{L^2(\T^m)}=1.
\]
However, by \eqref{equ-lem-add-freq-3}, we know that $\mathbf{1}_F(f_k+c_k)\to 0$ strongly in $L^2(\T^m)$. By the uniqueness of the limit, we obtain
\begin{align}
\label{equ-lem-add-freq-4}
f+c=0 \quad \text{on  } F.
\end{align}
For every $h\in \T^m$, introduce a new function $g=f(\cdot+h)-f$ then 
\begin{equation*}
g(x)=\sum_{k\in\Lambda}\widehat{f}(k)(e^{ik\cdot(x+h)}-e^{ik\cdot x})=\sum_{k\in\Lambda}\widehat{f}(k)(e^{ik\cdot h}-1)e^{ik\cdot x},
\end{equation*}
which means that $\mathrm{supp\, }\widehat{g}\subset \Lambda$. In addition, \eqref{equ-lem-add-freq-4} implies
\begin{align}
\label{equ-lem-add-freq-5}   
g=0 \text{ on } F\cap (F-h) \text{ for all }h\in \T^m.
\end{align}
However, by the assumption \eqref{equ-lem-add-freq-1}, together with the Plancherel theorem, for $|h|<\delta_0$, we have
\begin{align}
\label{equ-lem-add-freq-6} 
\|g\|^2_{L^2(\T^m)}&=\sum_{k\in\Lambda}|\widehat{f}(k)(e^{ikh}-1)|^2\notag\\
&\leq C(F)\int_{F\cap (F-h)}|\sum_{k\in\Lambda}\widehat{f}(k)(e^{ikh}-1)e^{ik\cdot x}|^2\d x \notag\\
&=C(F)\int_{F\cap (F-h)}|g|^2\d x.
\end{align}
Combining \eqref{equ-lem-add-freq-5}-\eqref{equ-lem-add-freq-6}, we obtain
$g=0$ on $\T^m$ for all $|h|\leq \delta_0$. Thus, all Fourier coefficients of $g$ vanish, namely
$$
0=\widehat{g}(k)=(e^{ik\cdot h}-1)\widehat{f}(k), \quad \forall k\in \Lambda, |h|\leq \delta_0.
$$
Noting $0\notin \Lambda$, this implies that $\widehat{f}(k)=0$ for all $k\in \Lambda$, then $f=0$ on $\T^m$. This, together with \eqref{equ-lem-add-freq-4}, gives $f+c\equiv 0$, which leads to a contradiction with the fact $\|f+c\|_{L^2(\T^m)}=1$.  
\end{proof}
Now we are in a position to finish the proof of Theorem \ref{thm-1}.
\begin{proof}[Proof of Theorem \ref{thm-1}]
Comparing the high-frequency estimate \eqref{equ-G-10} and the full observability \eqref{equ-ob-intro}, we need to add the finite low-frequency part $\{(k,p(k)): |k|\leq N\}$. Let $\Lambda=\{(k,p(k)): |k|>N\}$, where $N$ is the same as in Lemma \ref{lem-high-freq}. Thanks to Lemma \ref{lem: perturbative}, there exists $\delta=\delta(G)>0$ such that for $\forall \{a_k\}\in l^2$ and $\forall h\in(-\delta,\delta)$, we have
\begin{equation*}
\sum_{(k,p(k))\in\Lambda}|a_k|^2\leq C(G)\iint_{G\cap (G-h)}\Big|\sum_{(k,p(k))\in\Lambda}a_ke^{i(kx+p(k)t)}\Big|^2\d x \d t. 
\end{equation*}
Given a single point $(k_0,p(k_0))$ with $|k_0|\leq N$, let $\Lambda_1=\Lambda\cup \{(k_0,p(k_0))\}$. Applying Lemma \ref{lem-add-freq}, we obtain
\begin{equation*}
\sum_{(k,p(k))\in\Lambda_1}|a_k|^2\leq C'(G)\iint_{G}\Big|\sum_{(k,p(k))\in\Lambda_1}a_ke^{i(kx+p(k)t)}\Big|^2\d x \d t. 
\end{equation*}
It is easy to see that \eqref{equ-514-1} still holds true, i.e.,there exists $\widetilde{C}(G)<\frac{1}{2C'(G)}$ and $\delta_1\in (0,\delta)$ such that
\begin{equation*}
\iint_{G\backslash (G-h)}\Big|\sum_{(k,p(k))\in\Lambda_1}a_ke^{i(kx+p(k)t)}\Big|^2\d x \d t\leq \widetilde{C}(G)\sum_{(k,p(k))\in\Lambda_1}|a_k|^2, \quad \forall \{a_k\}\in l^2 , \forall h\in (-\delta_1,\delta_1).
\end{equation*}
This implies that 
\begin{equation*}
\sum_{(k,p(k))\in\Lambda_1}|a_k|^2\leq 2C'(G)\iint_{G\cap (G-h)}\Big|\sum_{(k,p(k))\in\Lambda_1}a_ke^{i(kx+p(k)t)}\Big|^2\d x \d t, \quad \forall \{a_k\}\in l^2,\forall h\in (-\delta_1,\delta_1).
\end{equation*}
Since $\#\{(k,p(k)): |k|\leq N\}=2N+1$, we repeat this procedure $2N+1$ times, we obtain \eqref{equ-G-1}.
\end{proof}
\begin{remark}
In fact, the proof of Proposition \ref{thm-1} can be given by a direct application of Miheev's theorem. Lemma \ref{lem-str} shows that the lattice points $\{(k,p(k)):k\in \Z)\}$ is a $\Lambda(4)$ set, in the terminology of Rudin \cite{Rudin60}. Then according to the Miheev theorem (see \cite{HavinTheuncertain1994} for the one dimension case, and \cite{BonamiUCP2006} for higher dimensions), if $G$ is a measurable set in $\T^2$ with positive measure, we have
\begin{align}\label{equ-G-9}
\left\|\sum_{k\in \Z}a_ke^{i(kx+p(k)t)}\right\|_{L^2(\T^2)}\leq C(G)\left\|\sum_{k\in \Z}a_ke^{i(kx+p(k)t)}\right\|_{L^2(G)}.
\end{align}
But the left hand side of \eqref{equ-G-9} is a constant times $(\sum_{k\in \Z}|a_k|^2)^{1/2}$, thus we conclude \eqref{equ-G-1}. This proof uses the deep Miheev theorem in a black box way. 
\end{remark}
Compared to the preceding framework, applying Miheev’s theorem is more restrictive. As we will present in Section \ref{sec: linearized system}, the preceding framework offers greater flexibility and can be adapted to a broader range of settings.

\subsection{Proof of Theorem \ref{thm-spacetime-a} and consequences}\label{sec: conclusion}
In this subsection, we complete the proof of Theorem \ref{thm-spacetime-a}.
\begin{proof}[Proof of Theorem \ref{thm-spacetime-a}]
The first statement is a direct consequence of H\"older inequality and Lemma \ref{lem-Stri-LinftyL2}. Indeed,
$$
\int_0^T\int_\T |a||e^{itP(D)}u_0|^2\d x \d t\leq \|a\|_{L^1_x(\T;L^\infty_t[0,T])}\|e^{itP(D)}u_0\|^2_{L^\infty_x(\T;L^2_t[0,T])}\leq C\|u_0\|^2_{L^2(\T)}.
$$
To show \eqref{equ-55-15-intro}, for every $j\geq 1$, consider the set $\Omega_j:=\{(x,t)\in \T\times[0,T]: |a(x,t)|\geq \frac{1}{j} \}$.
Since $\|a\|_{L^1_x(\T;L^\infty_t[0,T])}>0$, there exists a $j_0>0$ such that the 2-dimensional Lebesgue measure $|\Omega_{j_0}|>0$. Thanks to Theorem \ref{thm-1}, we have $\|u_0\|^2_{L^2(\T)}\lesssim  \iint_{\Omega_{j_0}} |e^{itP(D)}u_0|^2\d x \d t, \forall u_0\in L^2(\T)$.
This, together with the definition of $\Omega_j$, gives
\begin{align*}
\|u_0\|^2_{L^2(\T)}\lesssim j_0  \iint_{\Omega_{j_0}} |a||e^{itP(D)}u_0|^2\d x \d t\lesssim j_0  \int_0^T\int_\T |a||e^{itP(D)}u_0|^2\d x \d t  
\end{align*}
for all $u_0\in L^2(\T)$ as required.
\end{proof}
The following corollary provides another type of observability. The following lemma gives an improvement of \eqref{equ-55-15-intro}, which will be used in the study of exponential decay in the last section.
\begin{lemma}\label{lem-ob-compact-pertur}
Let $T>0$ and let $\mathcal{A}$ be a precompact subset in $L^1_x(\T;L^\infty_t([0,T]))$ such that for some constant $a_0>0$,
\begin{align}\label{equ-55-16}
\|a\|_{L^1_x(\T;L^\infty_t([0,T]))}\geq a_0, \quad \forall a\in \mathcal{A}.    
\end{align}
Then there exists a constant $C>0$ such that
\begin{align}\label{equ-55-17}
 \|u_0\|^2_{L^2(\T)}\leq C\int_0^T\int_\T|a||e^{itP(D)}u_0|^2\d x \d t, \quad \forall a\in \mathcal{A}, u_0\in L^2(\T).  
\end{align}
\end{lemma}
\begin{proof}
We argue by contradiction. Suppose that there exist sequences $u_{0n}\in L^2(\T)$ and $a_k\in \mathcal{A}$ such that
\begin{align}\label{equ-55-18}
\|u_{0n}\|_{L^2(\T)}=1, \quad \int_0^T\int_\T|a_k||e^{itP(D)}u_{0n}|^2\d x \d t\to 0, \quad \text{ as } n\to \infty.    
\end{align}
Since $a_k\in \mathcal{A}$ and $\mathcal{A}$ is precompact in $L^1_x(\T;L^\infty_t([0,T]))$, there exists a subsequence, still denoted by $a_k$, and $a\in L^1_x(\T;L^\infty_t[0,T])$ such that $a_k\to a$ in $L^1_x(\T;L^\infty_t([0,T]))$. This, together with the Strichartz estimate in Lemma \ref{lem-Stri-LinftyL2}, gives 
\begin{align}\label{equ-55-19}
\int_0^T\int_\T|a_k-a||e^{itP(D)}u_{0n}|^2\d x \d t&\leq \|a_k-a\|_{L^1_x(\T;L^\infty_t([0,T]))}\|e^{itP(D)}u_{0n}\|^2_{L^\infty_x(\T;L^2_t([0,T]))}\nonumber\\
&\lesssim  \|a_k-a\|_{L^1_x(\T;L^\infty_t([0,T]))}\|u_{0n}\|^2_{L^2(\T)} \to 0  
\end{align}
as $n\to \infty$. Combining \eqref{equ-55-18}-\eqref{equ-55-19}, we obtain
\begin{align}\label{equ-55-20}
 \int_0^T\int_\T|a||e^{itP(D)}u_{0n}|^2\d x \d t\to 0, \quad \text{ as } n\to \infty.     
\end{align}
Thanks to \eqref{equ-55-16}, we have $\|a_k\|_{L^1_x(\T;L^\infty_t([0,T]))}\geq a_0$. This implies $\|a\|_{L^1_x(\T;L^\infty_t([0,T]))}\geq a_0$. Thus by \eqref{equ-55-15-intro} and \eqref{equ-55-20}, we infer
$$
\|u_{0n}\|^2_{L^2(\T)}\leq C\int_0^T\int_\T |a||e^{itP(D)}u_{0n}|^2\d x \d t\to 0, \quad \text{ as } n\to \infty.     
$$
This leads to a contradiction with $\|u_{0n}\|_{L^2(\T)}=1$ for all $n\geq 1$.
\end{proof}

\section{Control of KdV equation in the rough setting}\label{sec: rough control of KdV}
This section is devoted to proving the main theorem \ref{thm: CKDV-N-intro}. We first analyze the linearized equation $\partial_tv+\partial_x^3v=\mathcal{L}(h)\mathbf{1}_{E_T\times F}$ and construct the control operator $\mathcal{K}$ (defined in \eqref{eq: defi-op-K}) in Subsection \ref{sec: linearized system}. Armed with the well-constructed $\mathcal{K}$, we finish the proof of Theorem \ref{thm: CKDV-N-intro} by a fixed point process in Bougain spaces in Subsection \ref{sec: nonlinear case}.
\subsection{Linearized system}\label{sec: linearized system}
 In general, if $u_0,u_1\in L^2_M(\T)$, let $\widetilde{u}_0=u_0-2M\pi\in L^2_0(\T)$ and $\widetilde{u}_1=u_1-2M\pi\in L^2_0(\T)$. It suffices to prove: $\exists h\in L^2([0,T]\times\T)$ such that the KdV equation
\begin{equation*}
\partial_tu+\partial_x^3u+u\partial_xu+2M\pi\partial_xu=\mathcal{L}(h)\mathbf{1}_{E_T\times F},
\end{equation*}
has a unique solution $u\in C([0,T];L^2(\T))$ satisfying that
\begin{equation*}
u(0,x)=\widetilde{u}_0(x)\in L^2_0(\T),u(T,x)=\widetilde{u}_1(x)\in L^2_0(\T).
\end{equation*}
This motivates us to consider the linearized system 
\begin{equation*}
\partial_tu+\partial_x^3u+2M\pi\partial_xu=\mathcal{L}(h)\mathbf{1}_{E_T\times F},
\end{equation*}

Since the extra $2M\pi\partial_x$ term does not effect the analysis, for the ease of the notation, for now on we only concentrate on the case $M= 0$:
\begin{equation}\label{eq: linear ckdv}
\partial_tv+\partial_x^3v=\mathcal{L}(h)\mathbf{1}_{E_T\times F},
\end{equation}
where the operator $\mathcal{L}$ is defined in \eqref{eq: defi-f} and $h\in L^2([0,T]\times\T)$. To establish the observability in Proposition \ref{prop: twist ob}, we present high-frequency estimates in Section \ref{sec: high-fre-est-kdv} and low-frequency analysis (especially augmented observability) in Section \ref{sec: low-fre-est-kdv}, equipped with the nice properties of $\mathcal{L}$ in Section \ref{sec: property of L}. We finish this part by constructing the control operator in Proposition \ref{prop: lckdv-result}.
\subsubsection{Basic properties for the operator $\mathcal{L}$}\label{sec: property of L}
In this sequel, we present some key properties of the operator $\mathcal{L}$, which we use later. We have the following lemmas.
\begin{lemma}\label{lem: 0-avg-op}
$\mathcal{L}:L^2(\T)\rightarrow L^2_0(\T)$ is a linear, bounded, self-adjoint operator. Moreover, for any $\psi(t)\in L^2((0,T))$, $\mathcal{L}\psi(t)\mathbf{1}_F(x)=\psi(t)\mathbf{1}_F(x)\mathcal{L}$. For any $h\in L^2([0,T]\times\T)$, we have $\|\mathcal{L}(h)\|_{L^2(E_T\times F)}=\|\mathcal{L}(h)\|_{L^2(E_T\times \T)}$.
\end{lemma}
The proof of Lemma \ref{lem: 0-avg-op} follows by direct computation and we choose to omit its details here. Next lemma concerns the matrix representation of $\mathcal{L}$ under the orthogonal basis $\{e^{ikx}\}_{k\in \Z}$.
\begin{lemma}
For any $h\in L^2_0(\T)$, $\mathcal{L}(h)$ can be represented by
\begin{equation}\label{eq: L(k,l)-defi-original}
\mathcal{L}(h)=\sum_{k,l\in\Z}L(k,l)\widehat{h}(k)e^{ilx},
\end{equation}
where $\widehat{h}(k)=\frac{1}{2\pi}\int_{\T}h(x)e^{-ikx}\d x$ and $L(k,l):=\frac{1}{2\pi}\int_{\T}\mathcal{L}(e^{ikx})e^{-ilx}\d x$. Moreover, we have the following properties
\begin{gather}
L(k,l)=\widehat{g}(l-k)-2\pi\widehat{g}(-k)\widehat{g}(l),\label{eq: coeff-op-L}\\
2\pi\sum_{l\in\Z}|L(k,l)|^2=\|\mathcal{L}(e^{ikx})\|^2_{L^2(\T)}=\frac{1}{|F|}(1-4\pi^2|\widehat{g}(k)|^2),\label{eq: g-k-leq1}\\
2\pi\sum_{l\in\Z}L(k,l)\overline{L(m,l)}=\poscals{\mathcal{L}(e^{ikx})}{\mathcal{L}(e^{imx})}_{L^2(\T)}=\frac{2\pi}{|F|}\left(\widehat{g}(m-k)-2\pi\overline{\widehat{g}(k)}\widehat{g}(m)\right).\label{eq: inner-coeff-km}
\end{gather}
\end{lemma}
\begin{proof}
For any $h\in L^2_0(\T)$, we expand $h$ into the Fourier series $h(x)=\sum_{k\in\Z}\widehat{h}(k)e^{ikx}$, where $\widehat{h}(k)=\frac{1}{2\pi}\int_{\T}h(x)e^{-ikx}\d x$. In particular, since $h\in L^2_0(\T)$, we have $h_0\equiv0$.  Then by the linearity of the operator $\mathcal{L}$, we know
\begin{equation}
\mathcal{L}(h)=\sum_{k\in\Z}\widehat{h}(k)\mathcal{L}(e^{ikx})=\sum_{k,l\in\Z}L(k,l)\widehat{h}(k)e^{ilx},
\end{equation}
where $L(k,l):=\frac{1}{2\pi}\int_{\T}\mathcal{L}(e^{ikx})e^{-ilx}\d x$. More precisely, by definition,
\begin{align*}
\int_{\T}\mathcal{L}(e^{ikx})e^{-ilx}\d x&=\frac{1}{|F|}\int_{\T}\mathbf{1}_{F}(x)\left(e^{ikx}-\frac{1}{|F|}\int_{\T}\mathbf{1}_F(y)e^{iky}\d y\right)e^{-ilx}\d x\\
&=\int_{\T}g(x)e^{i(k-l)x}\d x-\int_{\T}g(y)e^{iky}\d y\int_{\T}g(x)e^{-ilx}\d x.
\end{align*}
Recall that $g(x):=\frac{1}{|F|}\mathbf{1}_F(x),\forall x \in\T$. Then, we conclude \eqref{eq: coeff-op-L} as
\begin{equation*}
L(k,l)=\widehat{g}(l-k)-2\pi\widehat{g}(-k)\widehat{g}(l).
\end{equation*}
We consider the inner product $\frac{1}{2\pi}\poscals{\mathcal{L}(e^{ikx})}{e^{ilx}}_{L^2(\T)}$ next. Indeed, applying the Plancherel Theorem, we derive the following two identities:
\begin{equation*}
2\pi\sum_{l\in\Z}|L(k,l)|^2=\|\mathcal{L}(e^{ikx})\|^2_{L^2(\T)},\;\;\;
2\pi\sum_{l\in\Z}L(k,l)\overline{L(m,l)}=\poscals{\mathcal{L}(e^{ikx})}{\mathcal{L}(e^{imx})}_{L^2(\T)}.
\end{equation*}
In addition, we have more explicit formulas for these two identities above. Indeed, direct computation yields that
\begin{align*}
&\poscals{\mathcal{L}(e^{ikx})}{\mathcal{L}(e^{imx})}_{L^2(\T)}\\
=&\frac{1}{|F|^2}\int_{\T}\mathbf{1}_{F}(x)\left(e^{ikx}-\frac{1}{|F|}\int_{\T}\mathbf{1}_F(y)e^{iky}\d y\right)\overline{\left(e^{imx}-\frac{1}{|F|}\int_{\T}\mathbf{1}_F(y)e^{imy}\d y\right)}\d x\\
=&\frac{1}{|F|^2}\int_{\T}\left(\mathbf{1}_{F}(x)e^{i(k-m)x}-2\pi\overline{\widehat{g}(k)}\mathbf{1}_{F}(x)e^{-imx}-2\pi\widehat{g}(m)\mathbf{1}_{F}(x)e^{ikx}+4\pi^2\widehat{g}(m)\overline{\widehat{g}(k)}\right)\d x\\
=&\frac{2\pi}{|F|}\left(\widehat{g}(m-k)-2\pi\overline{\widehat{g}(k)}\widehat{g}(m)\right).
\end{align*}
Here we use several times the fact that $\widehat{g}(-k)=\overline{\widehat{g}(k)}$. In particular, if $k=m$, we have
\begin{equation*} 
\|\mathcal{L}(e^{ikx})\|_{L^2(\T)}^2=\frac{1}{|F|}\left(2\pi\widehat{g}(0)-4\pi^2|\widehat{g}(k)|^2\right)
=\frac{1}{|F|}(1-4\pi^2|\widehat{g}(k)|^2).   
\end{equation*}
Thus the proof is complete.
\end{proof}
The next lemma establishes the coercive estimate of $\|\mathcal{L}(e^{ikx})\|_{L^2(\T)}$.
\begin{lemma}\label{lem: coercive-L_k}
There exists a constant $\delta=\delta(F)>0$ such that 
\begin{equation}\label{eq: coercive-ineq}
\|\mathcal{L}(e^{ikx})\|_{L^2(\T)}^2>\delta>0,\forall k\in\Z,k\neq0.
\end{equation}
\end{lemma}
\begin{proof}
We first claim that $\|\mathcal{L}(e^{ikx})\|_{L^2(\T)}^2>0$, for any $k\in\Z,k\neq0$. Indeed, according to \eqref{eq: g-k-leq1}, we have $\|\mathcal{L}(e^{ikx})\|_{L^2(\T)}^2=\frac{1}{|F|}(1-4\pi^2|\widehat{g}(k)|^2)$. Using the explicit form of $\widehat{g}(k)$, we compute
\begin{equation*}
4\pi^2|\widehat{g}(k)|^2=\left|\int_{\T}\frac{1}{|F|}\mathbf{1}_F(x)e^{-ikx}\d x\right|^2\leq \left(\int_{\T}\left|\frac{1}{|F|}\mathbf{1}_F(x)e^{-ikx}\right|\d x\right)^2=1.    
\end{equation*}
Since $g=\frac{1}{|F|}\mathbf{1}_F$ is real-valued and positive, $\frac{1}{|F|}\mathbf{1}_F(x)e^{-ikx}$ cannot be a constant multiple of $\mathbf{1}_F(x)$ on $\T$. Therefore, the equality cannot hold, and we obtain the inequality 
\begin{equation*}
\|\mathcal{L}(e^{ikx})\|_{L^2(\T)}^2=\frac{1}{|F|}(1-4\pi^2|\widehat{g}(k)|^2)>0.
\end{equation*}
In addition, due to the Riemann--Lebesgue Lemma, we know that $|\widehat{g}(k)|\to0$ as $|k|\to\infty$, which implies that $\|\mathcal{L}(e^{ikx})\|_{L^2(\T)}^2\to\frac{1}{|F|}$ as $|k|\to\infty$. Then there exists a universal constant $\delta=\delta(F)>0$ such that $\|\mathcal{L}(e^{ikx})\|_{L^2(\T)}^2>\delta$ for $\forall k\in\Z,k\neq0$.
\end{proof}

We introduce the adjoint system associated with \eqref{eq: linear ckdv} as follows. Let $w$ be the solution to:
\begin{equation}\label{eq: adjoint system}
(\partial_t+\partial_x^3)w=0,\;w(0,x)=w_0(x).
\end{equation}
Let $S(t)$ be the linear unitary semi-group generated by $-\partial_x^3$. Then the solution $w$ to \eqref{eq: adjoint system} can be denoted as $w(t)=S(t)w_0$. We aim to prove the following ``twisted" observability inequality.
\begin{proposition}\label{prop: twist ob}
There exists a constant $C=C(E_T\times F)>0$ such that for any $\varphi\in L^2_0(\T)$, we have the following observability inequality:
\begin{equation}\label{eq: twist-ob}
\|\varphi\|^2_{L^2(\T)}\leq C\iint_{E_T\times F}|\mathcal{L}(S(t)\varphi)|^2\d x\d t.
\end{equation}
\end{proposition}
Our proof of Proposition \ref{prop: twist ob} is based on the high-frequency/low-frequency scheme that we used in Section \ref{sec: Observability from space-time measurable sets}. Let $\varphi(x)=\sum_{k\neq0}\widehat{\varphi}(k)e^{ikx}$. Then the solution $w(t):=S(t)\varphi$ satisfies \eqref{eq: adjoint system} with $w(0,x)=\varphi(x)$, and has the Fourier expansion $w(t,x)=\sum_{k\neq0}\widehat{\varphi}(k)e^{ik^3t}e^{ikx}$. Therefore, we deduce that $\mathcal{L}(w)=\sum_{k\neq0,l\in\Z}e^{ik^3t+ilx}L(k,l)\widehat{\varphi}(k)=\mathcal{L}_H(w)+\mathcal{L}_L(w)$, where
\begin{gather}
 \mathcal{L}_H(w):= \sum_{(k,l)\in Q^c}e^{i(k^3,  l)(t, x)} L(k, l)\widehat{\varphi}(k)\label{eq: hf-part-Lw}
\end{gather}
denotes the high-frequency part, $Q:=[-N_0,N_0]\setminus\{0\}\times\Z$. The remainder term, $\mathcal{L}_L(w)$, 
denotes the low-frequency part.
In the sequel, we deal with the high-frequency part and low-frequency part, respectively.
\subsubsection{High frequency estimates}\label{sec: high-fre-est-kdv}
\begin{proposition}\label{prop: hf-est-kdv}
Let $\varphi(x)\in L^2_0(\T)$ with $\varphi(x)=\sum_{k\neq0}\widehat{\varphi}(k)e^{ikx}$. Then, there exist a positive integer  $N_0$ and a constant $C_{H}=C_H(E_T,\delta)>0$ such that 
\begin{equation}\label{eq: high-frequency-est-kdv}
\sum_{|k|>N_0} |\widehat{\varphi}(k)|^2
 \leq C_H \int_{\T}\int_{E_T}|\sum_{(k,l)\in Q^c}e^{i(k^3,  l)(t, x)} L(k, l)\widehat{\varphi}(k)|^2\d t\d x,
\end{equation}
where $Q:=[-N_0,N_0]\setminus\{0\}\times\Z$.
\end{proposition}
\begin{proof}
We first simplify the right-hand side of \eqref{eq: high-frequency-est-kdv} 
\begin{align*}
\|\mathcal{L}_H(w)\|^2_{L^2(E_T\times\T)}&=\sum_{(k,l)\in Q^c,(m,l')\in Q^c}\int_{\T}\int_{E_T}e^{i(k^3-m^3)t}e^{i(l-l')x}L(k,l)\overline{L(m,l')}\widehat{\varphi}(k)\overline{\widehat{\varphi}(m)}\d t\d x\\
&=2\pi\sum_{|k|>N_0,|m|>N_0,l\in\Z}\int_{E_T}e^{i(k^3-m^3)t}L(k,l)\overline{L(m,l)}\widehat{\varphi}(k)\overline{\widehat{\varphi}(m)}\d t\\
&=I_0+I_H
\end{align*}
where $I_0:=2\pi|E_T|\sum_{|k|>N_0,l\in\Z}|L(k,l)|^2|\widehat{\varphi}(k)|^2$ and 
\[
I_{H}:=2\pi\sum_{k\neq m,|k|>N_0,|m|>N_0,l\in\Z}\int_{E_T}e^{i(k^3-m^3)t}L(k,l)\overline{L(m,l)}\widehat{\varphi}(k)\overline{\widehat{\varphi}(m)}\d t.
\]
Recall that there are uniform estimates (see \eqref{eq: g-k-leq1} and Lemma \ref{lem: coercive-L_k}):
\begin{gather*}
  \delta \leq \|\mathcal{L}(e^{ikx})\|_{L^2(\T)}^2\leq \frac{1}{|F|}, \forall k\neq0, \mbox{ or equivalently, }
    \delta  \leq  \sum_l |L(k, l)|^2 \leq \frac{1}{|F|}, \forall k\neq0.
\end{gather*}
Then, we derive that 
\begin{align*}
I_0=2\pi|E_T|\sum_{|k|>N_0} \sum_{l\in\Z}|L(k, l) \hat \varphi(k)|^2 
 &\geq 2\pi|E_T|\delta \sum_{|k|>N_0} |\widehat{\varphi}(k)|^2 .
\end{align*}

Now let us bound $|I_H|$ by $\sum_{|k|>N_0} |\widehat{\varphi}(k)|^2$. Without loss of generality, we assume that $E_T\subset\T$. We rewrite $I_H$ in the following form:
\begin{align*}
I_{H}=2\pi\sum_{k\neq m,|k|>N_0,|m|>N_0,l\in\Z}\widehat{\mathbf{1}_{E_T}}(m^3-k^3)\left( L(k,l)\overline{L(m,l)}\right)\widehat{\varphi}(k)\overline{\widehat{\varphi}(m)}.
\end{align*}
Using the Plancherel theorem, $\sum_{\alpha\in\Z}|\widehat{\mathbf{1}_{E_T}}(\alpha)|^2=|E_T|<\infty$. Thus, there exists $N_0\in\N^*$ sufficiently large such that 
\begin{equation*}
\left(\sum_{|\alpha|>N_0,\alpha\in\Z}|\widehat{\mathbf{1}_{E_T}}(\alpha)|^2\right)^{\frac{1}{2}}\leq \frac{|E_T||F|\delta}{20}.
\end{equation*}
Since $\mathcal{L}$ is self-adjoint, we know that $L(k,l)=\overline{L(l,k)}$, for any $(k,l)\in\Z^2$. Hence,
\[
\sum_{k}|L(k,l)|^2=\sum_{k}|\overline{L(k,l)}|^2=\sum_{k}|L(l,k)|^2.
\]
For $|k|>N_0$ and $|m|>N_0$, by \eqref{eq: inner-coeff-km} and \eqref{eq: g-k-leq1}, 
\begin{equation*}
|2\pi \sum_{l\in\Z}L(k,l)\overline{L(m,l)}|\leq \left|\poscals{\mathcal{L}(e^{ikx})}{\mathcal{L}(e^{imx})}_{L^2(\T)}\right|\leq \frac{1}{|F|}.
\end{equation*}
Similarly as in Lemma \ref{lem-high-freq}, by the Cauchy--Schwarz inequality, we obtain
\begin{align*}
&\left|\sum_{k\neq m,|k|>N_0,|m|>N_0}\widehat{\mathbf{1}_{E_T}}(m^3-k^3)\widehat{\varphi}(k)\overline{\widehat{\varphi}(m)}\poscals{\mathcal{L}(e^{ikx})}{\mathcal{L}(e^{imx})}_{L^2(\T)}\right|\\
\leq&
\frac{1}{|F|}\left(\sum_{k\neq m,|k|,|m|>N_0}|\widehat{\varphi}(k)\overline{\widehat{\varphi}(m)}|^2\right)^{\frac{1}{2}}\left(\Theta\sum_{|\alpha|>N_0,\alpha\in\Z}|\widehat{\mathbf{1}_{E_T}}(\alpha)|^2\right)^{\frac{1}{2}},
\end{align*}
where $\Theta$ is similar to \eqref{equ-G-5}, defined by $\Theta=\sup_{\alpha\in \Z^2\backslash \{0\}}\# \Big\{(k,l)\in \Z^2: l^3-k^3=\alpha\Big\}$. Since
\[
|l^3-k^3|=|l-k||l^2+kl+k^2|\geq \frac{1}{2}|l-k|(k^2+l^2),
\]
combining with $|k|>N_0,|m|>N_0,k\neq m$, we know that 
\begin{equation*}
|m^3-k^3|>\frac{1}{2}|m-k|(k^2+m^2)>N_0^2\geq N_0.
\end{equation*}
According to the proof of Lemma \ref{lem-str}, $\Theta\leq 2$.
Then, we deduce that
\begin{equation}\label{eq: k-l-small}
\left|2\pi\sum_{k\neq l,|k|>N_0,|m|>N_0}\sum_{l\in\Z}L(k,l)\overline{L(m,l)}\widehat{\varphi}(k)\overline{\widehat{\varphi}(m)}\widehat{\mathbf{1}_{E_T}}(m^3-k^3)\right|\leq \frac{|E_T|\delta}{10}\sum_{|k|>N_0}|\widehat{\varphi}(k)|^2.
\end{equation}
Therefore, we derive that
\begin{equation*}
\|\mathcal{L}_H(w)\|^2_{L^2(E_T\times\T)}\geq 2\pi|E_T|\delta \sum_{|k|>N_0} |\widehat{\varphi}(k)|^2-\frac{|E_T|\delta}{10}\sum_{|k|>N_0}|\widehat{\varphi}(k)|^2>\pi|E_T|\delta \sum_{|k|>N_0} |\widehat{\varphi}(k)|^2, 
\end{equation*}
which implies that $C_H(E_T,\delta)=\frac{1}{\pi|E_T|\delta}>0$
\[
\sum_{|k|>N_0} |\widehat{\varphi}(k)|^2
 \leq C_H(E_T,\delta)\int_{\T}\int_{E_T}|\sum_{(k,l)\in Q^c}e^{i(k^3,  l)(t, x)} L(k, l)\widehat{\varphi}(k)|^2\d t\d x.
\]
This completes the proof.
\end{proof}

\subsubsection{Uniform high-frequency estimates under translations}
Recall $E_T-h:=\{t-h: t\in E_T\}$.
\begin{lemma}
Let $N_0$ and $Q$ be the same as in Proposition \ref{prop: hf-est-kdv}. Then there exists a constant $\varepsilon_0$ depending on $E_T,F$ such that
\begin{equation}\label{eq: stability-hf-est-kdv}
\sum_{|k|>N_0} |\widehat{\varphi}(k)|^2
 \leq 2C_H \int_{\T}\int_{E_T\cap(E_T-h)}|\sum_{(k,l)\in Q^c}e^{i(k^3,  l)(t, x)} L(k, l)\widehat{\varphi}(k)|^2\d t\d x,
\end{equation}
for all $|h|<\varepsilon_0$.
\end{lemma}
\begin{proof}
Due to the fact that $E_T=\left(E_T\cap(E_T-h)\right)\bigcup\left(E_T\setminus(E_T-h)\right)$, it suffices to demonstrate that 
\begin{equation}\label{eq: stab-est-2}
2C_H \int_{\T}\int_{E_T\setminus(E_T-h)}|\sum_{(k,l)\in Q^c}e^{i(k^3,  l)(t, x)} L(k, l)\widehat{\varphi}(k)|^2\d t\d x\leq \sum_{|k|>N_0} |\widehat{\varphi}(k)|^2,
\end{equation}
holds for all $|h|<\varepsilon_0$. Indeed, applying Proposition \ref{prop: hf-est-kdv}, we have
\begin{align*}
2\sum_{|k|>N_0} |\widehat{\varphi}(k)|^2
 &\leq2C_H \int_{\T}\int_{E_T}|\sum_{(k,l)\in Q^c}e^{i(k^3,  l)(t, x)} L(k, l)\widehat{\varphi}(k)|^2\d t\d x\\
 &\leq \sum_{|k|>N_0} |\widehat{\varphi}(k)|^2+2C_H \int_{\T}\int_{E_T\cap(E_T-h)}|\sum_{(k,l)\in Q^c}e^{i(k^3,  l)(t, x)} L(k, l)\widehat{\varphi}(k)|^2\d t\d x.
\end{align*}
Then we obtain the desired inequality \eqref{eq: stability-hf-est-kdv}. Now we turn to prove \eqref{eq: stab-est-2}. We denote
\begin{align*}
J_1=&\left|\int_{\T}\int_{E_T\setminus(E_T-h)}|\sum_{(k,l)\in Q^c}e^{i(k^3,  l)(t, x)} L(k, l)\widehat{\varphi}(k)|^2\d t\d x\right|\\
=&2\pi\left|\sum_{(k,l)\in Q^c,(m,l)\in Q^c}\int_{E_T\setminus(E_T-h)}e^{i(k^3-m^3)t}L(k,l)\overline{L(m,l)}\widehat{\varphi}(k)\overline{\widehat{\varphi}(m)}\d t\right|.
\end{align*}
Similarly as in Lemma \ref{lem-high-freq}, by the Cauchy--Schwarz inequality, we obtain 
\begin{align*}
|J_1|&\leq4\pi\left(\sum_{|k|,|m|>N_0}|\sum_{l\in\Z}L(k,l)\overline{L(m,l)}\widehat{\varphi}(k)\overline{\widehat{\varphi}(m)}|^2\right)^{\frac{1}{2}}\left(\sum_{|\alpha|>N_0,\alpha\in\Z}|\widehat{\mathbf{1}_{E_T\setminus(E_T-h)}}(\alpha)|^2\right)^{\frac{1}{2}}\\
&\leq 4\pi\left(\sum_{|k|,|m|>N_0}|\widehat{\varphi}(k)\overline{\widehat{\varphi}(m)}|^2\sum_{l\in\Z}|L(k,l)|^2\sum_{l\in\Z}|L(m,l)|^2\right)^{\frac{1}{2}}\left(\sum_{|\alpha|>N_0,\alpha\in\Z}|\widehat{\mathbf{1}_{E_T\setminus(E_T-h)}}(\alpha)|^2\right)^{\frac{1}{2}}.
\end{align*}
Thanks to Plancherel's theorem, we have 
\begin{equation*}
\left(\sum_{|\alpha|>N_0,\alpha\in\Z}|\widehat{\mathbf{1}_{E_T\setminus(E_T-h)}}(\alpha)|^2\right)^{\frac{1}{2}}\leq \left(\sum_{\alpha\in\Z}|\widehat{\mathbf{1}_{E_T\setminus(E_T-h)}}(\alpha)|^2\right)^{\frac{1}{2}}=\|\mathbf{1}_{E_T\setminus(E_T-h)}\|_{L^2}=|E_T\setminus(E_T-h)|^{\frac{1}{2}}.
\end{equation*}
Consequently, we obtain
\begin{equation*}
|J_1|\leq 4\pi|E_T\setminus(E_T-h)|^{\frac{1}{2}}\sum_{|k|>N_0,l\in\Z}|L(k,l)\widehat{\varphi}(k)|^2\leq \frac{2}{|F|}|E_T\setminus(E_T-h)|^{\frac{1}{2}}\sum_{|k|>N_0} |\widehat{\varphi}(k)|^2.
\end{equation*}
Armed with the preceding estimate for $J_1$, we derive that
\begin{equation*}
2C_H\int_{\T}\int_{E_T\setminus(E_T-h)}|\sum_{(k,l)\in Q^c}e^{i(k^3,  l)(t, x)} L(k, l)\widehat{\varphi}(k)|^2\d t\d x\leq \frac{4C_H}{|F|}|E_T\setminus(E_T-h)|^{\frac{1}{2}}\sum_{|k|>N_0} |\widehat{\varphi}(k)|^2.
\end{equation*}
Since the Lebesgue measure $|E_T\setminus(E_T-h)|\to0$ as $|h|\to0$, there exists a constant $\varepsilon_0=\varepsilon_0(E_T,F,C_H)>0$ such that $\frac{4C_H}{|F|}|E_T\setminus(E_T-h)|^{\frac{1}{2}}<1$. As a consequence, we prove the estimate \eqref{eq: stab-est-2}.
\end{proof}

\subsubsection{Analysis of low-frequency part}\label{sec: low-fre-est-kdv}
In the sequel, we focus on the low-frequency part of \eqref{eq: twist-ob}. It suffices to prove the following lemma on the augmented observability.
\begin{lemma}[Augmented observability]\label{lem-add-freq-kdv}
Let $Q$ be the same as in Proposition \ref{prop: hf-est-kdv}. Assume that the following inequality 
\begin{align*}
\sum_{|k|>N_0} |\widehat{\varphi}(k)|^2
 \leq 2C_H \int_{\T}\int_{E_T\cap(E_T-h)}|\sum_{(k,l)\in Q^c}e^{i(k^3,  l)(t, x)} L(k, l)\widehat{\varphi}(k)|^2\d t\d x,
\end{align*}
holds for all $\varphi\in L^2_0(\T)$ and all $|h|\leq \varepsilon_0$ with some $\varepsilon_0>0$.
Let $\lambda\in \{(k_0,l):0<|k_0|\leq N_0,l\in\Z\}\subset \Z^2 \backslash Q$. Then
there exists a constant $C'_H>0$ such that for all $\varphi\in L^2_0(\T)$, we have
\begin{align}\label{equ-lem-add-freq-kdv}
\sum_{|k|>N_0,\mbox{ or }k=k_0 } |\widehat{\varphi}(k)|^2
 \leq C'_H \int_{\T}\int_{E_T}|\sum_{(k,l)\in Q^c\cup\{\lambda\}}e^{i(k^3,  l)(t, x)} L(k, l)\widehat{\varphi}(k)|^2\d t\d x.
\end{align}
\end{lemma}
\begin{proof}
The proof is quite similar to Lemma \ref{lem-add-freq}. We argue by contradiction. Suppose that \eqref{equ-lem-add-freq-kdv} fails, then there exist a sequence $\{\varphi_n\}_{n\in\N^*}\subset L^2_0(\T)$ with $\mathrm{supp\,} \widehat{\varphi_n} \subset\{k\in\Z:|k|>N_0\}$ such that
\begin{gather}
|\widehat{\varphi_n}(k_0)|^2+\sum_{|k|>N_0} |\widehat{\varphi_n}(k)|^2=1, \label{eq: norm-1-kdv}\\
\int_{\T}\int_{E_T}|\sum_{(k,l)\in Q^c\cup\{\lambda\}}e^{i(k^3,  l)(t, x)} L(k, l)\widehat{\varphi_n}(k)|^2\d t\d x \to 0 \text{ as } n\to \infty.  \label{eq: 0-limit}
\end{gather}
It follows that there exist subsequences, still denoted by $\varphi_n$, such that $\varphi_n\rightharpoonup \varphi$ weakly in $L^2_0(\T)$ and $\widehat{\varphi_n}(k_0)\to c$ for some $\varphi\in L^2_0(\T)$ with $\mathrm{ supp\, } \widehat{\varphi}\subset\{k\in\Z:|k|>N_0\}$ and $c\in \C$. Thus $\varphi_n+\widehat{\varphi_n}(k_0)\rightharpoonup \varphi+c$ weakly in $L^2_0(\T)$. 

 In fact, we claim that $\{\varphi_n\}_{n\in\N}$ is a Cauchy sequence. Using the Plancherel theorem, together with $\mathrm{supp\,} \widehat{\varphi_n} \subset\{k\in\Z:|k|>N_0\}$ , we have
\[
\|\varphi_m-\varphi_n\|^2_{L^2_0(\T)}=\sum_{|k|>N_0}|\widehat{\varphi_m}(k)-\widehat{\varphi_n}(k)|^2.
\]
Applying the high-frequency estimate \eqref{eq: high-frequency-est-kdv}, we derive that
\begin{align*}
\sum_{|k|>N_0}|\widehat{\varphi_m}(k)-\widehat{\varphi_n}(k)|^2&\leq C_H \int_{\T}\int_{E_T}|\sum_{(k,l)\in Q^c}e^{i(k^3,  l)(t, x)} L(k, l)\left(\widehat{\varphi_m}(k)-\widehat{\varphi_n}(k)\right)|^2\d t\d x\\
&\leq 3C_H\int_{\T}\int_{E_T}|\sum_{(k,l)\in Q^c\cup\{\lambda\}}e^{i(k^3,  l)(t, x)} L(k, l)\widehat{\varphi_n}(k)|^2\d t\d x\\
&+3C_H\int_{\T}\int_{E_T}|\sum_{(k,l)\in Q^c\cup\{\lambda\}}e^{i(k^3,  l)(t, x)} L(k, l)\widehat{\varphi_m}(k)|^2\d t\d x\\
&+3C_H\int_{\T}\int_{E_T}|\sum_{l\in\Z}e^{i(k_0^3,  l)(t, x)} L(k_0, l)\left(\widehat{\varphi_m}(k_0)-\widehat{\varphi_n}(k_0)\right)|^2\d t\d x.
\end{align*}
Since $\{\widehat{\varphi_n}(k_0)\}_{n\in\N}$ is a Cauchy sequence, combining with \eqref{eq: 0-limit}, we know that
$\{\varphi_n\}_{n\in\N}$ is a Cauchy sequence in $L^2_0(\T)$. This implies that $\varphi_n\to\varphi$ strongly in $L^2_0(\T)$, and $\varphi_n+\widehat{\varphi_n}(k_0)\to \varphi+c$ strongly in $L^2_0(\T)$. Hence,
\begin{equation}\label{eq: norm-1-limit}
\|\varphi+ce^{ik_0x}\|^2_{L^2_0(\T)}=1,\mbox{ or equivalently, }|c|^2+\sum_{|k|>N_0}|\widehat{\varphi}(k)|^2=1.
\end{equation}
However, by \eqref{eq: 0-limit}, we know that 
\[
\mathbf{1}_{E_T}(t)\sum_{(k,l)\in Q^c\cup\{\lambda\}}e^{i(k^3,  l)(t, x)} L(k, l)\widehat{\varphi_n}(k)\to 0, \mbox{ strongly in } L^2(\T^2).
\]
By the uniqueness of the limit, we obtain
\begin{equation*}
\sum_{|k|>N_0,l\in\Z}e^{i(k^3,  l)(t, x)} L(k, l)\widehat{\varphi}(k)+\sum_{l\in\Z}ce^{i(k_0^3,  l)(t, x)} L(k_0, l)=0,\text{ on  } E_T\times\T.
\end{equation*}
Hence, we obtain for all $l\in\Z$
\begin{equation}\label{eq: 0-E_T}
\sum_{|k|>N_0}e^{i(k^3-k_0^3)t} L(k, l)\widehat{\varphi}(k)+cL(k_0, l)=0\quad  \text{ on  } E_T.
\end{equation}
Let $f_l(t)=\sum_{|k|>N_0}e^{i(k^3-k_0^3)t} L(k, l)\widehat{\varphi}(k)$ and $g_l=f_l(\cdot+h)-f_l$. Then
\begin{equation*}
g_l(t)=\sum_{|k|>N_0}e^{i(k^3-k_0^3)t} L(k, l)(e^{i(k^3-k_0^3)h}-1)\widehat{\varphi}(k),\mbox{ and }g_l=0\mbox{ on }E_T\cap(E_T-h).
\end{equation*}
Therefore, due to \eqref{eq: stability-hf-est-kdv}, for all $|h|\leq \varepsilon_0$, we have
\begin{align*}
&\sum_{|k|>N_0} |(e^{i(k^3-k_0^3)h}-1)\widehat{\varphi}(k)|^2\\
& \leq 2C_H \int_{\T}\int_{E_T\cap(E_T-h)}|\sum_{(k,l)\in Q^c}e^{i(k^3,  l)(t, x)} L(k, l)(e^{i(k^3-k_0^3)h}-1)\widehat{\varphi}(k)|^2\d t\d x\\
& =4\pi C_H \sum_{l\in\Z}\int_{E_T\cap(E_T-h)}|g_l(t)|^2\d t\equiv0.
\end{align*}
Thus, we have
$$
0=(e^{i(k^3-k_0^3)h}-1)\widehat{\varphi}(k), \quad \forall |k|>N_0, |h|\leq \varepsilon_0.
$$
Noting $|k_0|\leq N_0$, this implies that $e^{i(k^3-k_0^3)h}-1\neq0$ for all $|k|>N_0$, which ensures that $\widehat{\varphi}(k)=0,\forall |k|>N_0$. According to \eqref{eq: 0-E_T}, we now have
\[
cL(k_0, l)=0 \text{ on  } E_T,\forall l\in\Z.
\]
Since $\sum_{l\in\Z}|L(k_0, l)|^2>\delta>0$, then there exists $l_0\in\Z$ such that $L(k_0, l_0)\neq0$, which implies that $c\equiv0$. Together with $\widehat{\varphi}(k)=0,\forall |k|>N_0$, we have
\[
\varphi+ce^{ik_0x}\equiv0
\]
 which leads to a contradiction with the fact $\|\varphi+ce^{ik_0x}\|_{L_0^2(\T)}=1$.  
\end{proof}

\subsubsection{Construction of the control operator}\label{sec: construction of the control operator}
\begin{lemma}[Duality relation]\label{lem: duality relation}
Let $w$ be the solution to the adjoint system \eqref{eq: adjoint system}. Then, we have the following duality identity:
\begin{equation}\label{eq: dual identity}
\poscals{\mathcal{L}^2(h)\mathbf{1}_{E_T\times F}}{S(t)w_0}_{L^2([0,T]\times\T)}=\poscals{v(T,\cdot)}{S(T)w_0}_{L^2(\T)}-\poscals{v(0,\cdot)}{w_0}_{L^2(\T)},
\end{equation}
where $v$ satisfies the equation $\partial_tv+\partial_x^3v=\mathcal{L}^2(h)\mathbf{1}_{E_T\times F}$.
\end{lemma}
For any $\widetilde{h}\in L^2([0,T]\times\T)$, we set $\widetilde{v}$ be the unique solution to 
\begin{equation*}
\partial_t\widetilde{v}+\partial_x^3\widetilde{v}=\mathcal{L}(\widetilde{h})\mathbf{1}_{E_T\times F},\;\;\widetilde{v}|_{t=T}=0.
\end{equation*}
Then, we are able to define the range operator $\mathcal{R}: L^2([0,T]\times\T)\rightarrow L^2_0(\T)$ by $\mathcal{R}(\widetilde{h})=\widetilde{v}|_{t=0}$. Using this operator $\mathcal{R}$, we define the HUM operator $\Phi:L^2_0(\T)\rightarrow L^2_0(\T)$ via
\begin{equation}
\Phi(\varphi):=-\mathcal{R}\circ\mathcal{L}\circ S(t)(\varphi).
\end{equation}
This HUM operator $\Phi$ has the following property:
\begin{proposition}\label{prop: iso-L^2}
$\Phi$ is an isomorphism on $L^2_0(\T)$. Moreover, there exists a constant $C=C(E_T\times F)>0$ such that $\|\Phi\|_{\mathcal{B}(L^2_0(\T))}+\|\Phi^{-1}\|_{\mathcal{B}(L^2_0(\T))}\leq C$.
\end{proposition}
\begin{proof}
Let us define a continuous form $\alpha$ on $L^2_0(\T)\times L^2_0(\T)$ by $\alpha(\varphi_1,\varphi_2)=\poscals{\Phi(\varphi_1)}{\varphi_2}_{L^2(\T)}$. Then we check that $\alpha$ is coercive. For $\varphi\in L^2_0(\T)$, using the definition of $\Phi$
\begin{align*}
\alpha(\varphi,\varphi)=\poscals{\Phi(\varphi)}{\varphi}_{L^2(\T)}=-\poscals{\mathcal{R}\circ\mathcal{L}\circ S(t)(\varphi)}{\varphi}_{L^2(\T)}.
\end{align*}
Thanks to the duality identity \eqref{eq: dual identity}, we know that
\begin{equation*}
\poscals{\mathcal{L}(h)\mathbf{1}_{E_T\times F}}{S(t)\varphi}_{L^2([0,T]\times\T)}=-\poscals{v(0,\cdot)-S(-T)v(T,\cdot)}{\varphi}_{L^2(\T)},
\end{equation*}
where $v$ satisfies the equation $\partial_tv+\partial_x^3v=\mathcal{L}^2(h)\mathbf{1}_{E_T\times F}$. Let $\widetilde{v}$ be the unique solution to 
\begin{equation*}
\partial_t\widetilde{v}+\partial_x^3\widetilde{v}=\mathcal{L}^2(h)\mathbf{1}_{E_T\times F},\;\;\widetilde{v}(0,x)=v(0,\cdot)-S(-T)v(T,\cdot).
\end{equation*}
Then, we derive that $\widetilde{v}(T,\cdot)\equiv0$. Indeed, using Duhamel's formula, we can obtain $\widetilde{v}(T,\cdot)=0$ easily. Therefore, by the definition of $\mathcal{R}$, we obtain $\mathcal{R}(\mathcal{L}(h))=v(0,\cdot)-S(-T)v(T,\cdot)$, which implies that the duality identity is in the following form:
\begin{equation}\label{eq: duality-R}
\poscals{\mathcal{L}^2(h)\mathbf{1}_{E_T\times F}}{S(t)\varphi}_{L^2([0,T]\times\T)}=-\poscals{\mathcal{R}(\mathcal{L}(h))}{\varphi}_{L^2(\T)}.
\end{equation}
Due to \eqref{eq: duality-R}, we obtain an equivalent form of $\alpha(\varphi,\varphi)$:
\begin{align*}
\alpha(\varphi,\varphi)=-\poscals{\mathcal{R}\circ\mathcal{L}\circ S(t)(\varphi)}{\varphi}_{L^2(\T)}=\poscals{\mathcal{L}^2(S(t)\varphi)\mathbf{1}_{E_T\times F}}{S(t)\varphi}_{L^2([0,T]\times\T)}.
\end{align*}
Applying Proposition \ref{prop: twist ob}, and the observability inequality 
\begin{equation*}
\|w_0\|^2_{L^2(\T)}\leq C(E_T\times F)\iint_{E_T\times F}|\mathcal{L}(w)|^2\d x\d t,
\end{equation*}
we know that $\alpha$ is coercive, $\alpha(\varphi,\varphi)\geq \frac{1}{C(E_T\times F)}\|\varphi\|^2_{L^2(\T)}$. By the Lax-Milgram theorem, $\Phi$ is an isomorphism on $L^2_0(\T)$ and $\exists C=C(E_T,F)>0$ such that $\|\Phi\|_{\mathcal{B}(L^2_0(\T))}+\|\Phi^{-1}\|_{\mathcal{B}(L^2_0(\T))}\leq C$.
\end{proof}
Based on the HUM operator, we define the control operator $\mathcal{K}: L^2_0(\T)\times L^2_0(\T)\rightarrow L^2([0,T]; L^2_0(\T))$ via
\begin{equation}\label{eq: defi-op-K}
\mathcal{K}(\varphi_1,\varphi_2):=-\mathcal{L}\circ S(t)\circ\Phi^{-1}(\varphi_1-S(-T)\varphi_2).
\end{equation}
Then we have the following exact controllability result:
\begin{proposition}\label{prop: lckdv-result}
Let $T>0$. Then for any $v_0,v_1\in L^2_0(\T)$, there exists a control $h=\mathcal{K}(v_0,v_1)\in L^2([0,T]\times\T)$ such that the linearized KdV equation \eqref{eq: linear ckdv} has a unique solution $v\in C([0,T];L^2(\T))$ satisfying that $v(0,x)=v_0(x),v(T,x)=v_1(x)$. Moreover, there exists a constant $C_{\mathcal{K}}=C_{\mathcal{K}}(g,T,E_T,F)>0$ such that
\begin{equation}\label{eq: est-K-op}
\|\mathcal{K}(v_0,v_1)\|_{L^2([0,T]\times\T)}\leq C_{\mathcal{K}}\left(\|v_0\|_{L^2(\T)}+\|v_1\|_{L^2(\T)}\right).
\end{equation}
\end{proposition}

\subsection{Nonlinear case}\label{sec: nonlinear case}
We begin by preparing some elements for the proof of Theorem \ref{thm: CKDV-N-intro}. In the pioneering work \cite{Bourgain-93-kdv}, Bourgain introduced the Fourier restriction norms 
\begin{align*}
N_{s,b}(h):=\left(\sum_{k\in\Z}(1+|k|)^{2s}\int_{\R}|\widehat{h}(\tau,k)|^2(1+|\tau-k^3|)^{2b}\d \tau\right)^{\frac{1}{2}},\\
\widetilde{N}_{s,b}(h):=\left(\sum_{k\in\Z}(1+|k|)^{2s}\left(\int_{\R}|\widehat{h}(\tau,k)|(1+|\tau-k^3|)^{b}\d \tau\right)^2\right)^{\frac{1}{2}}.
\end{align*}
where $h\in L^2(\R\times\T)$ and $\widehat{h}$ denotes the Fourier transform of $h$ with respect to both $t$ and $x$ variables and $s,b\in\R$. Based on these two norms, we introduce the classical Bourgain spaces.
\begin{definition}[Bourgain spaces]\label{defi: bourgain space}
Given $s,b\in\R$, for a function $h\in L^2(\R\times\T)$, the Bourgain spaces associated to the KdV equation on $\T$ are defined by
\begin{gather*}
X^{s,b}:=\{h\in L^2(\R;H^s(\T)): N_{s,b}(h)<\infty\},\qquad
Y^{s,b}:=\{h\in L^2(\R;H^s(\T)): \widetilde{N}_{s,b}(h)<\infty\},\\
Z^{s,b}:=\{h\in X^{s,b}\cap Y^{s,b-\frac{1}{2}}: N_{s,b}(h)+\widetilde{N}_{s,b-\frac{1}{2}}(h)<\infty\}.  
\end{gather*}
It is clear that $X^{s,b},Y^{s,b},Z^{s,b}$ are all Hilbert spaces, endowed with the norms $\|\cdot\|_{X^{s,b}}=N_{s,b}(\cdot)$, $\|\cdot\|_{Y^{s,b}}=\widetilde{N}_{s,b}(\cdot)$, and $\|\cdot\|_{Z^{s,b}}=N_{s,b}(\cdot)+\widetilde{N}_{s,b}(\cdot)$, respectively. Let $X_T^{s,b}$ be the restriction space\footnote{Similarly for other two types of Bourgain spaces.} of $X^{s,b}$ equipped with the norm $\|h\|_{X_T^{s,b}}:=\inf\{\|\widetilde{h}\|_{X^{s,b}}: \widetilde{h}\in X^{s,b} \mbox{ with }\widetilde{h}=h \mbox{ on }[0,T]\times\T\}$.
\end{definition}
We include the basic properties of Bourgain spaces in Appendix \ref{sec: Basic properties of Bourgain spaces} for completeness. Recall that $S(t)=e^{-t\partial_x^3}$ is the unitary linear semi-group generated by $-\partial_x^3$.

Now we are in a position to prove Theorem \ref{thm: CKDV-N-intro}. 
\begin{proof}[Proof of Theorem \ref{thm: CKDV-N-intro}]
 Then, by Duhamel's formula, we rewrite the mild solution to \eqref{eq: CKdV-tori-intro} by
\begin{equation}\label{eq: duhamel}
u(t)=S(t)u_0+\int_0^tS(t-\tau)(\mathcal{L}(h)\mathbf{1}_{E_T\times F})(\tau)\d \tau-\int_0^tS(t-\tau)(u\partial_x u)(\tau)\d \tau.
\end{equation}
Let $v_1(x;u):=\int_0^TS(T-\tau)(u\partial_x u)(\tau)\d \tau$. By the estimate \eqref{eq: stri-est-b+} in Lemma \ref{lem: bourgain-basic-est} and Lemma \ref{lem: bourgain bilinear-est}, we deduce that
\begin{equation}\label{eq: v-1-L^2}
\|v_1\|_{L^2(\T)}\leq C\|u\partial_xu\|_{Z_T^{0,-\frac{1}{2}}}\leq CT^{\theta}\|u\|^2_{Z_T^{0,\frac{1}{2}}}.    
\end{equation}
According to Proposition \ref{prop: lckdv-result}, we set $v\in C([0,T];L^2(\T))$ to be the unique solution to
\begin{equation}\label{eq: linear part-nkdv}
\left\{
\begin{array}{l}
    \partial_tv+\partial_x^3v=\mathcal{L}(\mathcal{K}(u_0,u_1+v_1))\mathbf{1}_{E_T\times F}, \\
    v(0,x)=u_0(x),
\end{array}
\right.
\end{equation}
satisfying that $v(T,x)=u_1(x)+v_1(x)$, that is, in an equivalent integral equation form:
\begin{equation}\label{eq: v-int-form}
v(t)=S(t)u_0+\int_0^tS(t-\tau)(\mathcal{L}(\mathcal{K}(u_0,u_1+v_1))\mathbf{1}_{E_T\times F})(\tau)\d \tau, \quad v(T)=u_1+v_1.
\end{equation}
We plug \eqref{eq: v-int-form} into \eqref{eq: duhamel}. Then, $u(t)=v(t)-\int_0^tS(t-\tau)(u\partial_x u)(\tau)\d \tau$ satisfies the equation 
\begin{equation*}
\partial_tu+\partial_x^3u+u\partial_xu=\mathcal{L}(\mathcal{K}(u_0,u_1+v_1))\mathbf{1}_{E_T\times F}, \;    u(0,x)=u_0(x),
\end{equation*}
and $u(T)=u_1$, which implies that the map $\Psi: Z^{0,\frac{1}{2}}_T\to Z^{0,\frac{1}{2}}_T$ defined via
\begin{equation}\label{eq: defi-psi-op}
\Psi(u):=S(t)u_0+\int_0^tS(t-\tau)(\mathcal{L}(\mathcal{K}(u_0,u_1+v_1))\mathbf{1}_{E_T\times F})(\tau)\d \tau-\int_0^tS(t-\tau)(u\partial_x u)(\tau)\d \tau,
\end{equation}
has a fixed point. It suffices to show that $\Psi$ is a contraction map in a bounded ball in $Z^{0,\frac{1}{2}}_T$. We split the norm $\|\Psi(u)\|_{Z^{0,\frac{1}{2}}_T}$ into the following three parts:
\begin{align*}
\|\Psi(u)\|_{Z^{0,\frac{1}{2}}_T}&\leq \underbrace{\|S(t)u_0\|_{Z^{0,\frac{1}{2}}_T}}_{J_1}+\underbrace{\|\int_0^tS(t-\tau)(\mathcal{L}(\mathcal{K}(u_0,u_1+v_1))\mathbf{1}_{E_T\times F})(\tau)\d \tau\|_{Z^{0,\frac{1}{2}}_T}}_{J_2}\\
&+\underbrace{\|\int_0^tS(t-\tau)(u\partial_x u)(\tau)\d \tau\|_{Z^{0,\frac{1}{2}}_T}}_{J_3}
\end{align*}
By \eqref{eq: bound by Hs} and \eqref{eq: stri-est-Z} in Lemma \ref{lem: bourgain-basic-est}, we know that 
\begin{equation*}
J_1\leq C\|u_0\|_{L^2(\T)},\;\;J_3\leq C\|u\partial_xu\|_{Z^{0,-\frac{1}{2}}_T}.
\end{equation*}
Applying Lemma \ref{lem: bourgain bilinear-est}, we know that $\|u\partial_xu\|_{Z^{0,-\frac{1}{2}}_T}\leq CT^{\theta}\|u\|^2_{X_T^{0,\frac{1}{2}}}$, which yields that
\begin{equation*}
J_1\leq C\|u_0\|_{L^2(\T)},\;\;J_3\leq CT^{\theta}\|u\|^2_{Z^{0,\frac{1}{2}}}.
\end{equation*}
Now we focus on $J_2$. Applying Lemma \ref{lem: 0-avg-op} and Proposition \ref{prop: lckdv-result}, we know that
\[
\|\mathcal{L}(\mathcal{K}(u_0,u_1+v_1))\mathbf{1}_{E_T\times F}\|_{L^2([0,T]\times\T)}\leq C\left(\|u_0\|_{L^2(\T)}+\|u_1\|_{L^2(\T)}+\|v_1\|_{L^2(\T)}\right).
\]
Due to \eqref{eq: v-1-L^2}, we know that $\|v_1\|_{L^2(\T)}\leq CT^{\theta}\|u\|^2_{Z_T^{0,\frac{1}{2}}}$. Combining with \eqref{eq: stri-est-Z} in Lemma \ref{lem: bourgain-basic-est} again, we obtain
\begin{align*}
J_2&\leq C\|\mathcal{L}(\mathcal{K}(u_0,u_1+v_1))\mathbf{1}_{E_T\times F}\|_{Z_T^{0,-\frac{1}{2}}}\\
&\leq C\|\mathcal{L}(\mathcal{K}(u_0,u_1+v_1))\mathbf{1}_{E_T\times F}\|_{L^2([0,T]\times\T)}\\
&\leq C(\|u_0\|_{L^2(\T)}+\|u_1\|_{L^2(\T)}+T^{\theta}\|u\|^2_{Z_T^{0,\frac{1}{2}}}).
\end{align*}
In summary, there exists a constant $C_{\Psi}>0$, independent of $u_0$, such that
\[
\|\Psi(u)\|_{Z^{0,\frac{1}{2}}_T}\leq C_{\Psi}(\|u_0\|_{L^2(\T)}+\|u_1\|_{L^2(\T)}+\|u\|^2_{Z^{0,\frac{1}{2}}_T}).
\]
For $R>0$, let $B_R$ be the ball centered at zero with radius $R$, defined by
\[
B_R:=\{u\in Z_T^{0,\frac{1}{2}};\|u\|_{Z_T^{0,\frac{1}{2}}}<R\}.
\]
For any $u,\widetilde{u}\in B_R$, due to the definition of $\Psi$ in \eqref{eq: defi-psi-op}, we have
\begin{gather*}
\Psi(u)-\Psi(\widetilde{u})=\int_0^tS(t-\tau)\left(\mathcal{L}(\mathcal{K}(u_0,u_1+v_1)-\mathcal{K}(u_0,u_1+\widetilde{v}_1))\mathbf{1}_{E_T\times F}\right)(\tau)\d \tau\\
-\int_0^tS(t-\tau)(u\partial_x u-\widetilde{u}\partial_x\widetilde{u})(\tau)\d \tau
\end{gather*}
Then we compute the norm $\|\Psi(u)-\Psi(\widetilde{u})\|_{Z_T^{0,\frac{1}{2}}}$ as
\begin{align*}
\|\Psi(u)-\Psi(\widetilde{u})\|_{Z_T^{0,\frac{1}{2}}}&\leq C\|\mathcal{K}(u_0,u_1+v_1)-\mathcal{K}(u_0,u_1+\widetilde{v}_1)\|_{Z_T^{0,\frac{1}{2}}}\\
&\quad +C\|\int_0^tS(t-\tau)\frac{1}{2}\partial_x\left((u-\widetilde{u})(u+\widetilde{u})\right)\d\tau\|_{Z_T^{0,\frac{1}{2}}}\\
&\leq C\|v_1-\widetilde{v}_1\|_{Z_T^{0,\frac{1}{2}}}+C\|u-\widetilde{u}\|_{Z_T^{0,\frac{1}{2}}}\|u+\widetilde{u}\|_{Z_T^{0,\frac{1}{2}}}.
\end{align*}
Thanks to the formula $v_1(x;u)=\int_0^TS(T-\tau)(u\partial_x u)(\tau)\d \tau$, there exists a constant $C'_{\Psi}>0$, independent of $u_0$, such that
\begin{gather*}
\|\Psi(u)-\Psi(\widetilde{u})\|_{Z_T^{0,\frac{1}{2}}}
\leq C'_{\Psi}\|u-\widetilde{u}\|_{Z_T^{0,\frac{1}{2}}}\|u+\widetilde{u}\|_{Z_T^{0,\frac{1}{2}}}.
\end{gather*}
If we choose $R$ and $\delta$ such that $C_{\Psi}(\delta R+R^2)<R$ and $C'_{\Psi}R<\frac{1}{2}$ with $\|u_0\|_{L^2(\T)}+\|u_0\|_{L^2(\T)}<\delta$, then for any $u,\widetilde{u}\in B_R$, we have
\[
\|\Psi(u)\|_{Z_T^{0,\frac{1}{2}}}<R,\mbox{ and }\|\Psi(u)-\Psi(\widetilde{u})\|_{Z_T^{0,\frac{1}{2}}}\leq \frac{1}{2}\|u-\widetilde{u}\|_{Z_T^{0,\frac{1}{2}}}.
\]
Therefore, $\Psi$ is a contraction map and then the proof is complete.
\end{proof}

\vspace{4mm}

\section{Applications to exponential stabilization}\label{sec: damping}

In this final section, we provide an application of the observability inequalities obtained in Section \ref{sec: Observability from space-time measurable sets}. In fact, we shall establish a uniform exponential decay of solutions to a localized damping version of \eqref{equ-1}, at $L^2(\T)$-level norm as time goes to infinity. The problem is as follows. 
Let $a\in L^1_x(\T;L^\infty_t(0,\infty))$ and consider the damped dispersive equation on torus $\T$
\begin{align}\label{equ-decay}
 i\partial_tu+P(D)u+ia(t,x)u=0, (t,x)\in \R^+\times\T, \quad u|_{t=0}=u_0\in L^2(\T).   
\end{align}
Assume that $a(t,x)\geq 0$ for all $(t,x)\in \R^+\times\T$. Let $u$ be the solution of \eqref{equ-decay}. The standard energy estimate gives that for every $T>0$
\begin{align*}
\|u(0,\cdot)\|^2_{L^2(\T)}=\|u(T,\cdot)\|^2_{L^2(\T)}+2\int_0^T\int_\T a(t,x)|u(t,x)|^2\d x \d t.  
\end{align*}
The sign condition on $a(t,x)$ gives a damping effect, since it implies that 
\begin{align*}
    \|u(T,\cdot)\|_{L^2(\T)}\leq \|u_0\|_{L^2(\T)}.
\end{align*}
It is interesting to ask that whether an exponential decay like
\begin{align}\label{equ-decay-616-0}
    \|u(T,\cdot)\|_{L^2(\T)}\leq  e^{-\gamma t}\|u_0\|_{L^2(\T)}, \quad \forall u_0\in L^2(\T)
\end{align}
holds for some constant $\gamma>0$. It is proved in \cite{Russell-Zhang-96,PVZuazua2002,LRZ-2010} that \eqref{equ-decay-616-0} holds for the linear KdV equation (namely $p(k)=k^3$)  with a time-independent, localized damping term in the sense that
\begin{align}\label{equ-decay-616-1}
a(t,x) \equiv a(x)\geq a_0>0, \quad x\in \omega, 
\end{align}
where $\omega\subset \T$ is a nonempty open set. Similar results are obtained in \cite{BurqZworski2019} for the linear Schr\"odinger  (namely $p(k)=k^2$) when \eqref{equ-decay-616-1} holds with a measurable set $E$ of positive measure\footnote{Though their original result is stated in 2d case, the proof also works in 1d clearly.}.
However, little is known when $a$ depends on both time and space variables and $a$ has a positive lower bound only on a space-time positive measure set. Our aim here is to provide some results in this direction.

To state our result, we start with the following assumption ${\bf (A)}$ for a function $a\in L^1_x(\T;L^\infty_t(0,\infty))$.

\vspace{1mm}
\begin{itemize}
    \item[{\bf (A)}] there exists $T>0$ and $\alpha_0>0$ such that the set of restricitons of $a$ on $[(n-1)T,nT]\times \T$,
\begin{align}\label{equ-decay-Aset}
\mathcal{A}:=\{a_n\in L^1_x(\T;L^\infty_t[0,T]):a_n(t,x)= a(t,x)|_{\{(t,x)\in[(n-1)T,T]\times \T},n\geq1\},    
\end{align}
is precompact in $L^1_x(\T;L^\infty_t[0,T])$ and $\inf_{n\in\N}\|a_n\|_{L^1_x(\T;L^\infty_t[0,T])}\geq \alpha_0$.
\end{itemize}
Functions satisfying the assumption ${\bf (A)}$ exhibit ``controlled temporal behavior" when restricted to consecutive time blocks \([(n-1)T, nT] \times \mathbb{T}\). The precompactness of \( \mathcal{A} \) ensures ``temporal regularity" across blocks, while the lower bound \( \inf_{n\in\N}\|a_n\|_{L^1_x(\T;L^\infty_t[0,T])}\geq \alpha_0 > 0 \) prevents degeneracy (e.g., vanishing or decay to zero). The time-block length \( T \) is fixed and independent of \( n \).  
The choice of \( T \) depends on the function \( a \) (i.e., different \( a \) may require different \( T \)). We provide some concrete examples as follows.
\begin{itemize}
    \item [(1)] Time periodic case. Since $a|_{[(n-1)T, nT]}=a|_{[0,T]}$, \( \mathcal{A}\) is a singleton. It suffices to require that $\|a\|_{L^1_x(\T;L^\infty_t[0,T])}>0$.
    \item[(2)] Modulated Wave (non-periodic case). For fixed \( g \in C(\mathbb{T}) \) and $T>0$, let \( a(t,x) = \sin\left(2\pi t / T + \xi_n\right) g(x) \), where \( \{\xi_n\} \) is an arbitrary sequence in \([0, 2\pi)\). The set \(\mathcal{A}= \{ a_n=a|_{[(n-1)T, nT]} \}_{n\in\N^*} \) is bounded and equicontinuous in \( L^1_x(\T;L^\infty_t[0,T]) \). By Arzelà-Ascoli, it is precompact. Moreover, by definition,
       \[
       \|a_n\|_{L^1_x(\T;L^\infty_t[0,T])}\geq \int_{\mathbb{T}} |g(x)|  \d x \cdot c > 0,
    \]
    where \( c = \inf_{\xi\in[0,2\pi)} \sup_{t\in[0,T]} |\sin(2\pi t / T + \xi)| > 0 \).
Thus \( a \) satisfies the assumption ${\bf (A)}$ despite lacking strict periodicity. The phases \( \xi_n \) can wander arbitrarily, but the functional form ensures uniformity across blocks.  
\end{itemize}

\begin{theorem}[Exponential decay]\label{thm-decay}
Assume that $0\leq a\in L^1_x(\T;L^\infty_t(0,\infty))$ satisfies the assumption ${\bf (A)}$. Then there exist  constants $C,\gamma>0$ such that every solutions of \eqref{equ-decay} satisfies the exponential decay
$$
\|u(t,\cdot)\|_{L^2(\T)}\leq Ce^{-\gamma t}\|u_0\|^2_{L^2(\T)}, \quad \forall t\geq 0.
$$
\end{theorem}

\begin{figure}[htp]
    \centering
    \includegraphics[width=0.8\linewidth]{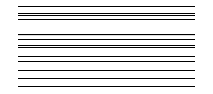}
    \caption{In previous references, the damping is posed on $(t,x)\in G=[0,\infty)\times E$, where $E\subset \T$ is open or measurable with positive measure, see the dark part of the figure.} 
    \label{fig:classicaldamping}
\end{figure}

\begin{figure}[htp]
    \centering
    \includegraphics[width=1\linewidth]{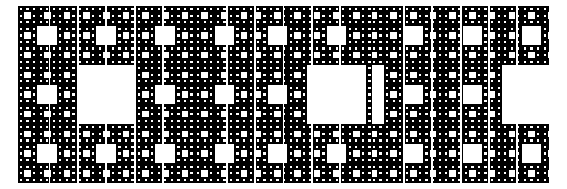}
    \caption{In our work, the damping region  $G$ is allowed to be the dark part of the figure, where   the horizontal direction representing the time axis and the vertical direction representing the spatial axis. Clearly, $G$ does not contains a subset of product structure as that in Fig. \ref{fig:classicaldamping}. }
    \label{fig:ourdamping}
\end{figure}

A particular interesting case is that $a(t,x)=a_0 1_{G}$ for some $a_0>0$ and $G$ is a subset set of $[0,\infty)\times \T$. In other words, the damping mechanism is posed on the set $G$. In Theorem \ref{thm-decay} we can take
\begin{align}\label{equ-618-0}
G=\bigcup_{n\geq 0} G_n, \quad G_n\subset [nT,(n+1)T)\times \T \mbox{ given by } 1_{G_n}(t,x)=1_{G_0}(\lfloor t+\xi_n \rfloor, x),    
\end{align}
where $\{\xi_n\}\subset [0,T]$ is a sequence, $\lfloor t+\xi_n \rfloor=t+\xi_n\mod T$, $G_0$ is a subset of $[0,T]\times \T$ with positive measure, see Fig. \ref{fig:ourdamping}. Indeed, one can check that $a(t,x)$ satisfies the assumption ${\bf (A)}$.

The rest of this subsection is devoted to the proof of Theorem \ref{thm-decay}. 
The proof consists of the following three steps.
\begin{itemize}
\item [\bf Step 1.] Nonhomogeneous Strichartz estimates. We establish a uniform resolvent estimate in Lemma \ref{lem-u-resolve}. Based on that, we prove $\|\int_0^te^{i(t-s)P(D)}f\d s\|_{L^\infty_x(\T;L^2_t[0,T])}\lesssim\|f\|_{L^1_x(\T;L^2_t[0,T])}$.
\item[\bf Step 2.] Uniform observability inequalities. Based on the Strichartz estimates in step 1 and observability inequalities proved in Section \ref{sec: Observability from space-time measurable sets}, we show that
\begin{align}\label{equ-618-1}
 \|u_0\|^2_{L^2(\T)}\lesssim\int_0^T\int_\T a_n(t,x)|u(t,x)|^2\d x \d t  
\end{align}
holds uniformly for all solution of \eqref{equ-decay} and all $a_n\in \mathcal{A}$, defined by \eqref{equ-decay-Aset}.

\item[\bf Step 3.] Exponential decay. Thanks to \eqref{equ-618-1}, we can show the contraction property of the $L^2$-norm, namely $\exists \alpha\in(0,1)$ such that $\|u(nT)\|_{L^2(\T)}\leq \alpha \|u((n-1)T)\|_{L^2(\T)}$ for $n\geq1$. After iteration, we obtain the desired exponential decay.
\end{itemize}
We start with a uniform resolvent estimate.
\begin{lemma}[Uniform resolvent estimate]\label{lem-u-resolve}
There exists a constant $C>0$ such that for all $z\in \C, | \text{\rm Im } z|\geq 1$ and $f\in L^1(\T)$,
\begin{align}
 \|(P(D)-z)^{-1}f\|_{L^\infty(\T)}\leq C\|f\|_{L^1(\T)}.   
\end{align}
\end{lemma}

\begin{proof}
Fix $f\in L^1(\T)$ and write $f(x)=\sum_{k\in \Z}c_ke^{ikx}$. If $z=\tau-is$, then $(P(D)-z)^{-1}f=\sum_{k\in \Z}(p(k)-\tau+i s)^{-1}c_ke^{ikx}$. Thus it suffices to prove that there exists a constant $C>0$ such that \begin{align}\label{equ-resolv-1}
\|\sum_{k\in \Z}(p(k)-\tau+i s)^{-1}c_ke^{ikx}\|_{L^\infty(\T)}\leq C\|f\|_{L^1(\T)}   
\end{align}
holds for all $\tau\in \R$ and $s\in \R$ with $|s|\geq 1$.

We first reduce  \eqref{equ-resolv-1} to the case $|\tau|\gg 1$. In fact, if $|\tau|\lesssim 1$, then 
$$
|(p(k)-\tau+i s)^{-1}|\lesssim (|p(k)-\tau|+|s|)^{-1}\lesssim (1+|k|)^{-d}
$$
which, together with the Sobolev embedding $H^1(\T)\hookrightarrow L^\infty(\T)$, shows that
\begin{align*}
&\|\sum_{k\in \Z}(p(k)-\tau+i s)^{-1}c_ke^{ikx}\|_{L^\infty(\T)}\lesssim\|\sum_{k\in \Z}(p(k)-\tau+i s)^{-1}c_ke^{ikx}\|_{H^1(\T)}\\
&\lesssim \left(\sum_{k\in \Z} \Big|(1+|k|)(p(k)-\tau+i s)^{-1}c_k\Big|^2 \right)^{1/2}\\
&\lesssim \left(\sum_{k\in \Z}(1+|k|)^{-2(d-1)}\right)^{1/2}\sup_{k\in \Z}|c_k| \lesssim \|f\|_{L^1(\T)}.
\end{align*}
Thus \eqref{equ-resolv-1} holds in this case.

Now assume that $|\tau|\gg 1$. Since $p$ is a monic polynomial with degree $d$, its dominated term is $k^d$ as $k\to \infty$. Thus there exists a large $K_0>0$ such that
\begin{align}\label{equ-resolv-2}
    \frac{1}{2}|k|^d\leq |p(k)|\leq \frac{3}{2}|k|^d  \quad \text{ for all } |k|\geq K_0.
\end{align}
Now we split the sum into high frequency part $|k|\geq K_0$ and low frequency part $|k|<K_0$,
\begin{multline*}
\|\sum_{k\in \Z}(p(k)-\tau+i s)^{-1}c_ke^{ikx}\|_{L^\infty(\T)}\\
\leq \|\sum_{|k|<K_0}(p(k)-\tau+i s)^{-1}c_ke^{ikx}\|_{L^\infty(\T)}+\|\sum_{|k|\geq K_0}(p(k)-\tau+i s)^{-1}c_ke^{ikx}\|_{L^\infty(\T)}.  
\end{multline*}
The low frequency part is easy, similarly to the argument in case $|\tau|\lesssim 1$, we have
$$
\|\sum_{|k|<K_0}(p(k)-\tau+i s)^{-1}c_ke^{ikx}\|_{L^\infty(\T)}\lesssim\left(\sum_{|k|< K_0} |(1+|k|)c_k|^2\right)^{1/2}\lesssim \sup_{|k|<K_0}|c_k|\lesssim \|f\|_{L^1(\T)}. 
$$
Thus, it remains to prove that  
\begin{align}\label{equ-resolv-3}
\|\sum_{|k|\geq K_0}(p(k)-\tau+i s)^{-1}c_ke^{ikx}\|_{L^\infty(\T)}\lesssim \|f\|_{L^1(\T)}, \quad |\tau|\gg 1, |s|\geq 1.    
\end{align}

In the sequel, we only prove \eqref{equ-resolv-3} when $\tau>0$, the proof in case $\tau<0$ is similar. We split further the sum in \eqref{equ-resolv-3} into two parts,
\begin{align*}
\sum_{|k|\geq K_0}(p(k)-\tau+i s)^{-1}c_ke^{ikx}=\sum_{|k|\geq K_0, |p(k)-\tau|>\frac{1}{4}|p(k)|}+\sum_{|k|\geq K_0, |p(k)-\tau|\leq \frac{1}{4}|p(k)|}.     
\end{align*}
For the first part, using \eqref{equ-resolv-2} and the Sobolev embedding $H^1(\T)\hookrightarrow L^\infty(\T)$ again, 
\begin{align}\label{equ-resolv-4}
\|\sum_{|k|\geq K_0,|p(k)-\tau|>\frac{1}{4}|p(k)|}(p(k)-\tau+i s)^{-1}c_ke^{ikx}\|_{L^\infty(\T)}\lesssim \|f\|_{L^1(\T)}.
\end{align}
To estimate the second part, we first claim that, if $0<\lambda\leq \frac{1}{2}\tau$, then
\begin{align}\label{equ-resolv-5}
\|\sum_{k\in S_\lambda} b_ke^{ikx}\|_{L^\infty(\T)}\lesssim \lambda \tau^{\frac{1}{d}-1}\sup_{k\in S_\lambda}|b_k|  
\end{align}
where $S_\lambda=\{k\in \Z:|k|\geq K_0,|p(k)-\tau|\leq \lambda\}$. Since $\|\sum_{k\in S_\lambda} b_ke^{ikx}\|_{L^\infty(\T)}\leq \sharp S_\lambda \cdot \sup_{k\in S_\lambda}|b_k|$,\eqref{equ-resolv-5} follows from the estimate
\begin{align}\label{equ-resolv-6}
 \sharp S_\lambda \lesssim \lambda \tau^{\frac{1}{d}-1}.   
\end{align}
When $d\geq 2$ is odd,  \eqref{equ-resolv-6} follows clearly from the statement: the distance of any two elements in $S_\lambda$ is   $\lesssim \lambda \tau^{\frac{1}{d}-1}$, namely
\begin{align}\label{equ-resolv-7}
 |k'-k''|\lesssim  \lambda \tau^{\frac{1}{d}-1}, \quad \forall  k',k''\in S_\lambda.  
\end{align}
Indeed, by the definition of $S_\lambda$, we have $k\sim \tau^{\frac{1}{d}}$ if $k\in S_\lambda$ and
\begin{align}\label{equ-resolv-8}
 |p(k')-p(k'')|\leq |p(k')-\tau|+|p(k'')-\tau|\leq 2\lambda, \quad \forall k',k''\in S_\lambda.    
\end{align}
Moreover, by the mean value theorem,
\begin{align}\label{equ-resolv-9}
   |p(k')-p(k'')|=|p'(\xi)||k'-k''|\sim \xi^{d-1}|k'-k''|\sim  \tau^{\frac{d-1}{d}}|k'-k''|.
\end{align}
Combining \eqref{equ-resolv-8} and \eqref{equ-resolv-9}, we obtain  \eqref{equ-resolv-7}. This proves \eqref{equ-resolv-6} if $d$ is odd.

When $d\geq 2$ is even, split $S_\lambda=S_\lambda^+\cup S_\lambda^-$, where
$$
S_\lambda^+(\text{resp. }S_\lambda^-)=\Big\{k\in \Z_{>0} (\text{resp. }\Z_{<0}):|k|\geq K_0,|p(k)-\tau|\leq \lambda\Big\}.
$$
Then one can proceeds similarly to show $\sharp S^+_\lambda, \sharp S^-_\lambda \lesssim \lambda \tau^{\frac{1}{d}-1}$, which also gives \eqref{equ-resolv-6}.

Now we use \eqref{equ-resolv-5} to bound the second part. It is easy to check that
$$
\{k\in \Z:|k|\geq K_0, |p(k)-\tau|\leq \frac{1}{4}|p(k)|\}
\subset \{k\in \Z:  |p(k)-\tau|\leq \frac{1}{3}\tau\}.
$$
Thus using  \eqref{equ-resolv-5} we have
\begin{align}\label{equ-resolv-10}
&\|\sum_{|k|\geq K_0,|p(k)-\tau|>\frac{1}{4}|p(k)|}(p(k)-\tau+i s)^{-1}c_ke^{ikx}\|_{L^\infty(\T)}\nonumber\\
&\lesssim \sum_{j\geq 0, 2^j\leq \frac{1}{3}\tau}\|\sum_{2^{j-1}\leq|p(k)-\tau|<2^j}(p(k)-\tau+i s)^{-1}c_ke^{ikx}\|_{L^\infty(\T)}\nonumber\\
&\lesssim \sum_{j\geq 0, 2^j\leq \frac{1}{3}\tau} \frac{2^j}{2^j+|s|}\tau^{\frac{1}{d}-1}\sup_{2^{j-1}\leq|p(k)-\tau|<2^j}|c_k|\nonumber\\
&\lesssim \sum_{j\geq 0} 2^{-j(1-\frac{1}{d})}\sup_{k\in \Z}|c_k|\lesssim \|f\|_{L^1(\T)}.
\end{align}
Combining \eqref{equ-resolv-4} and \eqref{equ-resolv-10} we obtain \eqref{equ-resolv-3}. Thus the proof is complete.
\end{proof}

\begin{lemma}\label{lem-decay-Stri-non}
Let $T>0$. Then there exists a constant $C>0$ such that for all $f\in L^1_x(\T;L^2_t[0,T])$,
$$
\|\int_0^te^{i(t-s)P(D)}f\d s\|_{L^\infty_x(\T;L^2_t[0,T])}\leq C\|f\|_{L^1_x(\T;L^2_t[0,T])}.
$$
\end{lemma}
\begin{proof}
Given $f\in L^1_x(\T;L^2_t[0,T])$. We still use $f$ to denote its zero extension to $(x,t)\in \T\times[0,\infty)$.
Let $v(t)=\int_0^te^{i(t-s)P(D)}f \d s, t\geq 0$. Then $v$ solves $i\partial_tv+P(D)v=f, t>0, 
v|_{t=0}=0.$
It remains to show that 
\begin{align}\label{decay-Stri-non-1}
\|v\|_{L^\infty_x(\T;L^2_t[0,T])}\leq C\|f\|_{L^1_x(\T;L^2_t[0,T])}.    
\end{align}
Let $V=ve^{-t}\mathbf{1}_{t>0}$ and $F=fe^{-t}\mathbf{1}_{0<t<T}$. Then using $v|_{t=0}=0$,  we have for all $t\in \R$,
\begin{align}\label{decay-Stri-non-2}
  i\partial_tV+P(D)V+iV=F.  
\end{align}
Taking Fourier transform of \eqref{decay-Stri-non-2} w.r.t $t$, we obtain for $\tau\in \R$,
\begin{align*}
  (P(D)+i-\tau)\widehat{V}(\tau)=\widehat{F}(\tau).  
\end{align*}
Applying Lemma \ref{lem-u-resolve} to $\widehat{V}$, we obtain $\|\widehat{V}(\tau)\|_{L^\infty(\T)}\leq C\|\widehat{F}(\tau)\|_{L^1(\T)}$.
Combining with Parseval's identity and Minkowski's inequality, we infer
\begin{align*}
\|v\|_{L^\infty_x(\T;L^2_t[0,T])} &\lesssim\|V\|_{L^\infty_x(\T;L^2_t(\R))}= \|\widehat{V}\|_{L^\infty_x(\T;L^2_\tau(\R))}\\
&\lesssim \|\widehat{V}\|_{L^2_\tau(\R;L^\infty_x(\T))}\leq C\|\widehat{F}\|_{L^2_\tau(\R;L^1_x(\T))}\\
&\lesssim \|\widehat{F}\|_{L^1_x(\T;L^2_\tau(\R))}= \|F\|_{L^1_x(\T;L^2_\tau(\R))}\lesssim \|f\|_{L^1_x(\T;L^2_t[0,T])}
\end{align*}
as desired.
\end{proof}

\begin{lemma}[Uniform observability]\label{lem-decay-ob}
Assume that $0\leq a\in L^1_x(\T;L^\infty_t(0,\infty))$ satisfies the assumption ${\bf (A)}$. Then there exists a constant $C>0$, independent of $n$, such that
\begin{equation*}
\|u_0\|^2_{L^2(\T)}\leq C\int_0^T\int_\T a_n(t,x)|u(t,x)|^2\d x \d t,
\end{equation*}
holds for all solution of \eqref{equ-decay} and all $a_n\in \mathcal{A}$, defined by \eqref{equ-decay-Aset}.
\end{lemma}
\begin{proof}
We argue by contradiction. Assume that there exist sequences $\{u_{0,l}\}_{l\in\N^*}\subset L^2(\T)$ and $\{b_l\}_{l\in\N^*}\subset \mathcal{A}$ such that
\begin{align}\label{equ-decay-ob-1}
\|u_{0,l}\|_{L^2(\T)}=1, \forall l\geq 1, \quad   \int_0^T\int_\T b_l(t,x)|u_l(t,x)|^2\d x \d t \to 0 \text{ as } l\to \infty,   
\end{align}
where $u_l$ solves the equation $i\partial_tu_l+P(D)u_l+ib_lu_l=0, u_l|_{t=0}=u_{0,l}$. By Duhamel's principle, $u_l(t)=e^{itP(D)}u_{0,l}-\int_0^te^{-i(t-s)P(D)}(b_lu_l)\d s$. Therefore, we have
\begin{align}\label{equ-decay-ob-4}
 \|b_l^{\frac{1}{2}}e^{itP(D)}u_{0,l}\|_{L^2([0,T]\times\T)}\leq \|b_l^{\frac{1}{2}}u_l\|_{L^2([0,T]\times\T)}+\|b_l^{\frac{1}{2}}\int_0^te^{-i(t-s)P(D)}(b_lu_l)\d s\|_{L^2([0,T]\times\T)}.   
\end{align}
Since $b_l\in \mathcal{A}$, $b_l^{\frac{1}{2}}$ is uniformly bounded in $L^2_x(\T;L^\infty_t[0,T])$. Thus, applying Lemma \ref{lem-decay-Stri-non}, we have
\begin{multline}\label{equ-decay-ob-5}
\|b_l^{\frac{1}{2}}\int_0^te^{-i(t-s)P(D)}(b_lu_l)\d s\|_{L^2([0,T]\times\T)}\leq C\|\int_0^te^{-i(t-s)P(D)}(b_lu_l)\d s\|_{L^\infty_x(\T;L_t^2[0,T])} \\
 \leq C\|b_lu_l \|_{L^1_x(\T;L^2_t[0,T])} \leq C\|b_l^{\frac{1}{2}} \|_{L^2_x(\T;L_t^\infty[0,T])}\|b_l^{\frac{1}{2}}u_l \|_{L^2_x(\T;L_t^2[0,T])}\leq C\|b_l^{\frac{1}{2}}u_l\|_{L^2([0,T]\times\T)}.
\end{multline} 
Plugging \eqref{equ-decay-ob-5} into \eqref{equ-decay-ob-4}, using  \eqref{equ-decay-ob-1}, we find
\begin{align}\label{equ-decay-ob-6}
\int_0^T\int_\T b_l(t,x)|e^{itP(D)}u_{0,l}|^2\d x \d t \to 0 \text{ as } l\to \infty.    
\end{align}
Since $b_l$ belongs to the compact set $\mathcal{A}$, according to Lemma \ref{lem-ob-compact-pertur}, we have
$$
\|u_{0,l}\|^2_{L^2(\T)}\leq C\int_0^T\int_\T b_l(t,x)|e^{itP(D)}u_{0,l}|^2\d x \d t,
$$
which, together with \eqref{equ-decay-ob-6}, implies that 
$\|u_{0,l}\|_{L^2(\T)}\to 0$ as $l\to \infty$, leading a contradiction.
\end{proof}
\begin{proof}[Proof of Theorem \ref{thm-decay}]
Consider the damped dispersive equation \eqref{equ-decay} in time intervals $[(n-1)T,nT]$, $n\geq1$, $i\partial_tu+P(D)u+ia(t,x)u=0$. The classical energy estimates imply
\begin{align}\label{equ-decay-proof-2}
\|u((n-1)T,\cdot)\|^2_{L^2(\T)}=\|u(nT,\cdot)\|^2_{L^2(\T)}+2\int_{(n-1)T}^{nT}\int_\T a(t,x)|u(t,x)|^2\d x \d t.  
\end{align}
Thanks to Lemma \ref{lem-decay-ob}, there exists a constant $C>0$ such that for all $n=1,2,\ldots,$
\begin{align}\label{equ-decay-proof-3}
 C\|u((n-1)T,\cdot)\|^2_{L^2(\T)}\leq 
 \int_{(n-1)T}^{nT}\int_\T a(t,x)|u(t,x)|^2\d x \d t.
\end{align}
It follows from \eqref{equ-decay-proof-2}-\eqref{equ-decay-proof-3} that for some $\alpha\in(0,1)$ such that
$$
\|u(nT,\cdot)\|^2_{L^2(\T)}\leq \alpha \|u((n-1)T,\cdot)\|^2_{L^2(\T)}, \quad n=1,2,\ldots.
$$
An iteration argument gives the desired exponential decay of $\|u(t,\cdot)\|_{L^2(\T)}$.
\end{proof}

\appendix
\section{Moment method and compactness-uniqueness method}
\subsection{Review of the moment method}\label{sec: moment method}
Here, we use a simplified model to provide a brief overview of the moment method. We consider the control problem of the linear KdV equation without any restriction.
\begin{equation}\label{eq: free-control-kdv}
(\partial_t+\partial_x^3)u=f,\mbox{ in }\T^2_{t,x},\;\;\;u|_{t=0}=u_0,u|_{t=T}=u_T.
\end{equation}
After writing the controlled equation \eqref{eq: free-control-kdv} into its Fourier modes, we have
\begin{gather*}
u_0(x)=\sum_{k\in\Z}\widehat{u}_0(k)e^{ikx},\;\;\;u_T(x)=\sum_{k\in\Z}\widehat{u}_T(k)e^{ikx},\;\;\;
u(t,x)=\sum_{k\in\Z}\widehat{u}_0(k)e^{ik^3t}e^{ikx}.
\end{gather*}
Therefore, we need to solve the moment problem in the following form:
\begin{equation}\label{eq: moment problem}
\widehat{u}_T(k)-\widehat{u}_0(k)e^{ik^3T}=\int_0^T\int_{\T}e^{ik^3(T-s)}f(t,x)e^{ikx}\d t\d x.    
\end{equation}
We aim to write our control function as
\begin{equation}\label{eq: moment-control-f}
f(t,x)=\sum_{k}f_k\phi_k(t)e^{-ikx}, 
\end{equation}
where $\phi_k$ satisfies the bi-orthogonal property:
\begin{equation}\label{eq: bi-orthogonal-prop}
\int_0^T\phi_k(t)e^{-il^3t}\d t=\delta_{kl},\forall k,l\in\Z,    
\end{equation}
which means that $\{\phi_k\}_{k\in\Z}$ forms a so-called \textit{bi-orthonormal family}. Indeed, thanks to \eqref{eq: moment-control-f}, we plug it into \eqref{eq: moment problem} and obtain
\[
\widehat{u}_T(k)-\widehat{u}_0(k)e^{ik^3T}=e^{ik^3T}\sum_{l\in\Z}f_l\int_0^T\int_{\T}\phi_l(s)e^{-ik^3s}e^{i(k-l)x}\d t\d x=2\pi e^{ik^3T}f_k.
\]
This algebraic equation provides the coefficients $f_k$, and we obtain our desired control function. 

In the present paper, we propose an alternative approach to the moment method—one that is better suited to rough control settings, based on techniques from harmonic analysis. 
\subsection{Brief revisit on compactness-uniqueness methods}\label{sec: compant-unique method}
In this part,  we briefly revisit how the compactness-uniqueness method applies to our mass-conserved KdV model. For further details, we refer interested readers to \cite{sun-kp,RS-kp2}. Basically, we aim to prove the following observability
\begin{equation}\label{eq: true ob}
\|u_0\|^2_{L^2(\T)}\leq C\int_0^T\int_{\T}|\mathcal{L}e^{it\partial_x^3}u_0|^2\d x\d t.
\end{equation}
The most crucial step is to establish the semiclassical observability for any well-spectrally localized initial states: $\exists h_0>0$ such that $\forall h\in(0,h_0)$ and $u_0\in L^2(\T)$
\begin{equation}\label{eq: semiclassical ob}
\|\chi(hD_x)u_0\|^2_{L^2(\T)}\leq C\int_0^T\int_{\T}|g(x)\chi(hD_x)u(t,x)|^2\d x\d t,
\end{equation}
where $\chi\in C^{\infty}_c(\R;[0,1])$ and $\supp\chi\subset[\frac{1}{2},2]$, and $u(t,x)$ is the solution to $h\partial_tu+(h\partial_x)^3u=0$. If \eqref{eq: semiclassical ob} holds, using the Littlewood-Paley partition of unity, 
\[
\sum_{k\in\Z}\chi_k(\xi)=1,\mbox{ with }\chi_k(\xi)=\chi(2^k\xi),\xi\neq0,
\]
based on a commutator estimate
\begin{equation}\label{eq: commutator est}
\int_0^T\|[\chi(hD_x),\mathcal{L}]e^{it\partial_x^3}u_0\|   ^2_{L^2(\T)}\d t\leq Ch^2\|u_0\|^2_{L^2(\T)}, 
\end{equation}
we are able to obtain the following weak observability 
\begin{equation}\label{eq: weak ob-com}
\|u_0\|^2_{L^2(\T)}\leq C\int_0^T\int_{\T}|g(x)u(t,x)|^2\d x\d t+C\|u_0\|^2_{H^{-1}}.
\end{equation}
Passing from the weak observability \eqref{eq: weak ob-com} to the true observability \eqref{eq: true ob}, we use the unique continuation property for eigenfunctions of $\partial_x^3$:
\begin{equation*}
\partial_x^3\phi=\lambda\phi,\phi\big|_{\supp g }=0\Rightarrow\phi=0.
\end{equation*}
To establish the semiclassical observability \eqref{eq: semiclassical ob} and the commutator estimate \eqref{eq: commutator est}, we normally rely on the symbolic $h$-pseudo-differential calculus. 

\section{Basic properties of Bourgain spaces}\label{sec: Basic properties of Bourgain spaces}
\begin{lemma}\label{lem: bourgain-basic-est}
For the Bourgain spaces defined in Definition \ref{defi: bourgain space}, we present some basic properties of Bourgain spaces here. Let $s,b\in\R$, and $T>0$ be given.
\begin{enumerate}
    \item If $b\leq b_1$, and $s\leq s_1$, then $X^{s_1,b_1}$ is continuously embedded in the space $X^{s,b}$.
    \item There exists a constant $C>0$ such that for any $\varphi\in H^s(\T)$, 
    \begin{equation}\label{eq: bound by Hs}
    \|S(t)\varphi\|_{X^{s,b}_T}\leq C\|\varphi\|_{H^s(\T)},\;\;\|S(t)\varphi\|_{Z^{s,b}_T}\leq C\|\varphi\|_{H^s(\T)}
    \end{equation}
    \item There exists a constant $C>0$ such that for any $h\in X^{s,b-1}_T$, 
    \begin{equation}\label{eq: stri-est-b+}
     \|\int_0^tS(t-t')h(t')\d t'\|_{X^{s,b}_T}\leq C\|h\|_{X^{s,b-1}_T}, \mbox{ provided that }b>\frac{1}{2}.   
    \end{equation}
    \item There exists a constant $C>0$ such that for any $h\in Z^{s,-\frac{1}{2}}_T$, 
    \begin{equation}\label{eq: stri-est-Z}
        \|\int_0^tS(t-t')h(t')\d t'\|_{Z^{s,\frac{1}{2}}_T}\leq C\|h\|_{Z^{s,-\frac{1}{2}}_T}.
    \end{equation}
\end{enumerate}
\end{lemma}
\begin{proof}
We refer to \cite{Tao-book,CKSTT-03} for its proof.  
\end{proof}
\begin{lemma}[Bilinear estimates]\label{lem: bourgain bilinear-est}
Let $s\geq 0$, $T\in(0,1)$, and $u,v\in X^{s,\frac{1}{2}}_T\cap L^2([0,T],L^2_0(\T))$. There exist some constants $\theta>0$ and $C>0$ independent of $T$ and $u,v$ such that
\begin{equation}\label{eq: bilinear-est}
\|\partial_x(uv)\|_{Z^{s,-\frac{1}{2}}_T}\leq CT^{\theta}\|u\|_{X^{s,\frac{1}{2}}_T}\|v\|_{X^{s,\frac{1}{2}}_T}
\end{equation}
\end{lemma}
\begin{proof}
This proof can be found in \cite{Bourgain-93-kdv} with $\theta=\frac{1}{12}$ (see also \cite{CKSTT-03}).
\end{proof}
    \normalem
    \bibliographystyle{alpha}
    \bibliography{ref}

\end{document}